\documentclass[a4paper,10pt, oneside]{article}
\usepackage{amsmath,amssymb,amsthm,bm,mathrsfs,graphicx,enumerate}
\usepackage{authblk}
\usepackage{float,listings}
\usepackage[titletoc, title]{appendix}
\usepackage[colorlinks,linkcolor=blue,citecolor=blue]{hyperref}
\usepackage{etoolbox}
\usepackage{longtable}
\usepackage{booktabs,makecell,multirow}
\usepackage[capitalise,sort]{cleveref}
\usepackage{cases,color}
\crefname{equation}{}{}
\crefname{lemma}{Lemma}{Lemmas}
\crefname{theorem}{Theorem}{Theorems}
\crefname{discr}{Discretization}{Discretizations}

\DeclareMathOperator{\D}{D}

\apptocmd{\sloppy}{\hbadness 10000\relax}{}{}

\newcommand{\dual}[1]{\langle {#1} \rangle}

\newcommand{\Dual}[1]{\left\langle {#1} \right\rangle}

\newcommand{\nm}[1]{\lVert {#1} \rVert}

\newcommand{\snm}[1]{\lvert {#1} \rvert}

\newcommand{\Snm}[1]{\left\lvert {#1} \right\rvert}
\newcommand{\ssnm}[1]
{
  \left\vert\kern-0.25ex
  \left\vert\kern-0.25ex
  \left\vert
  {#1}
  \right\vert\kern-0.25ex
  \right\vert\kern-0.25ex
  \right\vert
}

\makeatletter
\def\spher@harm#1{%
  \vbox{\hbox{%
    \offinterlineskip
    \valign{&\hb@xt@2\p@{\hss$##$\hss}\vskip.2ex\cr#1\crcr}%
  }\vskip-.36ex}%
}
\def\gshone{\spher@harm{.}}
\def\gshtwo{\spher@harm{.&.}}
\def\gshthree{\spher@harm{.&.&.}}
\let\gsh\spher@harm
\makeatother

\newtheorem{lemma}{Lemma}[section]
\newtheorem{remark}{Remark}[section]
\newtheorem{theorem}{Theorem}[section]

\makeatletter\def\@captype{table}\makeatother

\begin{document}
\title{
  \Large\bf{L1 scheme for solving an inverse problem subject to a fractional
  diffusion equation}
  \thanks{
    Binjie Li: School of Mathematics, Sichuan University, Chengdu 610064, China,
    (libinjie@scu.edu.cn); Xiaoping Xie: School of Mathematics, Sichuan
    University, Chengdu 610064, China, (xpxie@scu.edu.cn); Yubin Yan: Department
    of Mathematical and Physical Sciences, University of Chester, Thorn- ton
    Science Park, Pool Lane, Ince, CH2 4NU, UK, (y.yan@chester.ac.uk). Binjie Li
    was supported in part by the National Natural Science Foundation of China
    (NSFC) Grant No. 11901410; Xiaoping Xie was supported in part by the
    National Natural Science Foundation of China (NSFC) Grant No. 11771312; Dr.
    Yubin Yan is the corresponding author.
  }
}
\author{Binjie Li, Xiaoping Xie and Yubin Yan}

\date{}
\maketitle

\begin{abstract}
   This paper considers the temporal discretization of an inverse problem subject
  to a time fractional diffusion equation. Firstly, the convergence of the L1
  scheme is established with an arbitrary sectorial operator of spectral angle $
  < \pi/2 $, that is the resolvent set of this operator contains $ \{z\in\mathbb
  C\setminus\{0\}:\ \snm{\operatorname{Arg} z}< \theta\}$ for some $ \pi/2 <
  \theta < \pi $. The relationship between the time fractional order $\alpha \in
  (0, 1)$ and the constants in the error estimates is precisely characterized,
  revealing that the L1 scheme is robust as $ \alpha $ approaches $ 1 $. Then an
  inverse problem of a fractional diffusion equation is analyzed, and the
  convergence analysis of a temporal discretization of this inverse problem is
  given. Finally, numerical results are provided to confirm the theoretical
  results.
\end{abstract}


\section{Introduction}

Let $ 0 < T < \infty $ and let $ \Omega \subset \mathbb R^d $ ($d=1,2,3$) be a
bounded domain with Lipschitz continuous boundary. Assume that $ \mathcal A $ is
the realization of a second-order partial differential operator with homogeneous
Dirichlet boundary condition in $ L^2(\Omega) $. We consider the following
fractional diffusion equation:
\begin{equation}
  \label{eq:model}
  \D_{0+}^\alpha y (t) - \mathcal A y (t) = f (t),  \quad 0 < t \leqslant T,
  \quad\text{ with } y(0) = 0,
\end{equation}
where $ 0 < \alpha < 1 $, $ \D_{0+}^\alpha $ is a Riemann-Liouville fractional
differential operator of order $ \alpha $, and $ f $ is a given function.

The L1 scheme is one of the most popular numerical methods for fractional
diffusion equations. Lin and Xu \cite{Lin2007} analyzed the L1 scheme for the
fractional diffusion equation and obtained the temporal accuracy $
O(\tau^{2-\alpha})$ with $0< \alpha <1$, where $\tau$ denotes the time step
size. Sun and Wu \cite{sun2006fully} proposed the L1 scheme and derived temporal
accuracy $ O(\tau^{3-\alpha})$ with $1<\alpha <2$ for the fractional wave
equation. The analysis in the above two papers both assume that the underlying
solution is sufficiently smooth. However, Jin et~al.~\cite{Jin2016-L1} proved
that the L1 scheme is of only first-order temporal accuracy for fractional
diffusion equations with non-vanishing initial value, and Jin et~al.~\cite[Lemma
4.2]{Jin2018} derived only first-order temporal accuracy for an inhomogeneous
fractional equation. This phenomenon is caused by the well-known fact that the
solution of a fractional diffusion equation generally has singularity in time no
matter how smooth the data are, and it indicates that numerical analysis without
regularity restrictions on the solution is important for the fractional
diffusion equation. Recently, Yan et al.~\cite{Yan2018} proposed a modified L1
scheme for a fractional diffusion equation, which has $ (2-\alpha) $-order
temporal accuracy. For the L1 scheme with nonuniform grids, we refer the reader
to \cite{Stynes2017,Liao2018}; we also note that analyzing the L1 scheme with
nonuniform grids for a fractional diffusion equation with nonsmooth initial
value remains to be an open problem.


Although the sectorial operator is considered, the theoretical results in
\cite{Jin2016-L1,Yan2018} can not be applied to a fractional diffusion equation
with an arbitrary sectorial operator, since they require the spectral angle of
the sectorial operator not to be greater than $ \pi/4 $ (cf.~\cite[Remark
3.8]{Jin2016-L1}), that is the resolvent set of this operator must contain $ \{z
\in \mathbb C \setminus \{0\}:\, \snm{\operatorname{Arg} z} < 3\pi/4\} $. In our
work, the analysis is suitable for an arbitrary sectorial operator with spectral
angle $ < \pi/2 $.

As the fractional diffusion equation is an extension of the normal diffusion
equation, the solution of a fractional diffusion equation will naturally
converge to the solution of a normal diffusion equation as $ \alpha \to {1-} $,
and hence the L1 scheme is expected to be robust as $ \alpha \to {1-}
$. Recently, Huang et~al.~\cite{Huang2020} obtained an $ \alpha $-robust error
estimate for a multi-term fractional diffusion problem. However, to our best
knowledge, the $ \alpha $-robust convergence of the L1 scheme with an arbitrary
sectorial operator is not available in the literature. Here we note that the
constants in the error estimates in \cite{Lubich1996,Jin2016,Jin2016-L1,Yan2018}
all depend on $ \alpha $ and that the constants in the error estimates in
\cite{Jin2018-time-dependtn} will clearly blow up as $ \alpha \to {1-} $. This
motivates us to develop new techniques to analyze the convergence of the L1
scheme with an arbitrary sectorial operator and to investigate the robustness of
the L1 scheme as $ \alpha \to {1-} $.

The theory of inverse problems for differential equations has been extensively
developed within the framework of mathematical physics. One important class of
inverse problems for parabolic equations is to reconstruct the source term, the
initial value or the boundary conditions from the value of the solution at the
final time; see \cite{Prilepko2000,Samarskii2007}. The time fractional diffusion
equation is an extension of the normal diffusion equation, widely used to model
the physical phenomena with memory effect. Hence, this paper considers the
source term identification of a time fractional diffusion equation, based on the
value of the solution at the final time. For the related theoretical results, we
refer the reader to \cite{Jin2012,Liu2010,Murio2007,Tuan2016,Tuan2017,Wei2014}
and the references therein. We apply the famous Tikhonov regularization
technique to this inverse problem and establish the convergence of its temporal
discretization that uses the L1 scheme.

The main contributions of this paper are as follows:
\begin{enumerate}
  \item the convergence of the L1 scheme for solving time fractional diffusion
    equations with an arbitrary sectorial operator of spectral angle $ < \pi/2 $
    is established;
  \item the constants in the derived error estimates will not blow up as $
    \alpha \to {1-} $, which shows that the L1 scheme is robust as $ \alpha \to
    {1-} $;
  \item the convergence analysis of a temporally discrete inverse problem
    subject to a fractional diffusion equation is provided.
\end{enumerate}
Moreover, a feature of the error estimates in this paper is that they
immediately derive the corresponding error estimates of the backward Euler
scheme, by passing to the limit $ \alpha \to {1-} $.


Before concluding this section, we would also like to mention two important
algorithms for solving fractional diffusion equations. The first algorithm uses
the convolution quadrature proposed by Lubich \cite{Lubich1986,Lubich1988}.
Lubich et al.~\cite{Lubich1996,Cuesta2006} firstly used the convolution
quadrature to design numerical methods for fractional diffusion-wave equations,
and then Jin et al.~\cite{Jin2016,Jin-corre-2017} further developed these
algorithms. The second algorithm employs the Galerkin methods to discretize the
time fractional operators, which was firstly developed by McLean and
Mustapha
\cite{Mclean2009Convergence,Mustapha2009Discontinuous,Mustapha2011Piecewise,Mustapha2014A}.

The rest of the paper is organized as follows. Section 2 introduces some
conventions, the definitions of $ \mathcal A $ and $ \mathcal A^* $, the
Riemann-Liouville fractional operators and the mild solution theory of
linear fractional diffusion equations. Section 3 derives the convergence of the
L1 scheme. Section 4 investigates an inverse problem of a fractional diffusion
equation and establishes the convergence of a temporally discrete inverse
problem. Finally, Section 5 performs three numerical experiments to verify the
theoretical results.

\section{Preliminaries} 
\label{sec:pre}
Throughout this paper, we will use the following conventions: for each linear
vector space, the scalars are the complex numbers; $ H_0^1(\Omega) $ is
a standard complex-valued Sobolev space, and $ H^{-1}(\Omega) $ is the usual
dual space of $ H_0^1(\Omega) $; $ \mathcal L(L^2(\Omega)) $ is the space of all
bounded linear operators on $ L^2(\Omega) $; for a Banach space $\mathcal{B} $,
we use $ \dual{\cdot,\cdot}_\mathcal{B} $ to denote a duality paring between $
\mathcal{B}^* $ (the dual space of $ \mathcal{B} $) and $\mathcal{B} $; for a
Lebesgue measurable subset $ \mathcal D \subset \mathbb R^l $, $ 1 \leqslant l
\leqslant 4 $, $ \dual{p, q}_{\mathcal D} $ means the integral $ \int_{\mathcal
D} p \overline q $, where $ \overline q $ is the conjugate of $ q $; for a
function $ v $ defined on $ (0,T) $, by $ v(t-) $, $ 0 < t \leqslant T $, we
mean the limit $ \lim_{s \to t-} v(s) $; the notations $c_\times, d_\times, C_\times $ mean  some
positive constants and their values may differ
at each occurrence. In addition, for any $ 0 < \theta < \pi $, define
\begin{align}
  \Sigma_\theta &:= \{
    z \in \mathbb C \setminus \{0\}:
    -\theta < \operatorname{Arg} z < \theta
  \}, \label{eq:Sigma_theta-def} \\
  \Gamma_{\theta} &:= \{
    z \in \mathbb C \setminus \{0\}:\
    \snm{\operatorname{Arg} z} = \theta
  \} \cup \{0\} \label{eq:Upsilon-def} \\
  \Upsilon_\theta &:= \{
    z \in \Gamma_{\theta}:\
    -\pi \leqslant \Im z \leqslant \pi
  \}, \label{eq:Upsilon1-def}
\end{align}
where $ \Gamma_{\theta} $ and $ \Upsilon_\theta $ are so oriented that the
negative real axis is to their left. For the integral $ \int_{\Gamma_\theta} v
\, \mathrm{d}z $ or $ \int_{\Upsilon_\theta} v \, \mathrm{d}z $, if $ v $ has
singularity or is not defined at the origin, then $ \Gamma_\theta $ or $
\Upsilon_\theta $ should be deformed so that the origin is to its left; for
example, $ \Gamma_\theta $ is deformed to
\[
  \{
    z \in \mathbb C: \, \snm{z} > \epsilon, \,
    \snm{\operatorname{Arg} z} = \theta
  \} \cup \{
    z \in \mathbb C: \, \snm{z} = \epsilon, \,
    \snm{\operatorname{Arg} z} \leqslant \theta
  \},
\]
where $ 0 < \epsilon < \infty $.

\medskip\noindent{\bf Riemann-Liouville fractional calculus operators.} Assume
that $ -\infty \leqslant a < b \leqslant \infty $ and $ X $ is a Banach space.
For any $ \gamma > 0 $, define
\begin{align*}
  \left( \D_{a+}^{-\gamma}  v\right)(t) &:=
  \frac1{ \Gamma(\gamma) }
  \int_a^t (t-s)^{\gamma-1} v(s) \, \mathrm{d}s,
  \quad \text{a.e.}~t \in (a,b), \\
  \left(\D_{b-}^{-\gamma}  v\right)(t) &:=
  \frac1{ \Gamma(\gamma) }
  \int_t^b (s-t)^{\gamma-1} v(s) \, \mathrm{d}s,
  \quad\text{a.e.}~t \in (a,b),
\end{align*}
for all $ v \in L^1(a,b;X) $, where $ \Gamma(\cdot) $ is the gamma function. In
addition, let $ \D_{a+}^0 $ and $ \D_{b-}^0 $ be the identity operator on $
L^1(a,b;X) $. For $ j - 1 < \gamma \leqslant j $, $ j \in \mathbb N_{>0} $,
define
\begin{align*}
  \D_{a+}^\gamma v & := \D^j \, \D_{a+}^{\gamma-j}v, \\
  \D_{b-}^\gamma v & := (-\D)^j \, \D_{b-}^{\gamma-j}v,
\end{align*}
for all $ v \in L^1(a,b;X) $, where $ \D $ is the first-order differential
operator in the distribution sense.

\medskip\noindent{\bf Definitions of $ \mathcal A $ and $ \mathcal A^* $.} Let $
\mathcal A: H_0^1(\Omega) \to H^{-1}(\Omega) $ be a second-order partial
differential operator of the form
\[
  \mathcal A v := \sum_{i,j=1}^d \frac{\partial}{\partial_{x_i}}
  (a_{ij}(x) \frac{\partial}{\partial x_j} v ) +
   b(x) \cdot \nabla v + c(x)v,
  \quad \forall v \in H_0^1(\Omega),
\]
where $ a_{ij} \in L^\infty(\Omega) $, $  b \in [L^\infty(\Omega)]^d $ and $
c \in L^\infty(\Omega) $ are real-valued. Assume that $ \mathcal A:
H_0^1(\Omega) \to H^{-1}(\Omega) $ is a sectorial operator satisfying that
\begin{subequations}
\begin{numcases}{}
  \rho(\mathcal A) \supset \Sigma_{\omega_0},
  \label{eq:rho(A)} \\
  \nm{R(z,\mathcal A)}_{\mathcal L(L^2(\Omega))} \leqslant
  \mathcal M_0 \, |z|^{-1} \quad
  \forall z \in \Sigma_{\omega_0},
  \label{eq:R(z,A)} \\
  \dual{\mathcal Av,v}_{H_0^1(\Omega)}  \leqslant 0,
  \quad \forall v \in H_0^1(\Omega),
  \label{eq:A-positive}
\end{numcases}
\end{subequations}
where $ \rho(\mathcal A) $ is the resolvent set of $ \mathcal A $, $ \pi/2 <
\omega_0 < \pi $, $ R(z,\mathcal A) := (z-\mathcal A)^{-1} $, and $ \mathcal M_0
$ is a positive constant. Define the adjoint operator $ \mathcal A^*:
H_0^1(\Omega) \to H^{-1}(\Omega) $ of $ \mathcal A $ by that
\[
  \mathcal A^* v := \sum_{i,j=1}^d \frac{\partial}{\partial_{x_j}}
  (a_{ij}(x) \frac{\partial}{\partial x_i} v ) -
  \nabla \cdot ( b(x) v) + c(x)v,
  \quad \forall v \in H_0^1(\Omega).
\]
It is evident that
\[
  \dual{\mathcal A v, w}_{H_0^1(\Omega)} =
  \overline{
    \dual{\mathcal A^*w, v}_{H_0^1(\Omega)}
  }
  \quad\text{ for all } v, w \in H_0^1(\Omega).
\]

\medskip\noindent{\bf Solutions of the fractional diffusion equation.} For any $ t >
0 $, define
\begin{equation}
  \label{eq:E-def}
  E(t) := \frac1{2\pi i} \int_{\Gamma_{\omega_0}}
  e^{tz} R(z^\alpha,\mathcal A) \, \mathrm{d}z.
\end{equation}
By (\ref{eq:R(z,A)}), it is evident that $ E $ is an $ \mathcal L(L^2(\Omega))
$-valued analytic function on $ (0,\infty) $. Moreover, a direct computation
gives the following two estimates (cf.~Jin et al.~\cite{Jin2016}): for any $ t >
0 $,
\begin{align} 
  \nm{E(t)}_{\mathcal L(L^2(\Omega))}
  & \leqslant C_{\omega_0,\mathcal M_0} t^{\alpha-1},
  \label{eq:E} \\
  \nm{E'(t)}_{\mathcal L(L^2(\Omega))}
  & \leqslant C_{\omega_0,\mathcal M_0} t^{\alpha-2}.
  \label{eq:E'}
\end{align}
For any $ g \in L^1(0,T;L^2(\Omega)) $, we call
\begin{equation}
  \label{eq:Sg-l1}
  (Sg)(t) := (E*g) (t)  =  \int_0^t E(t-s) g(s) \, \mathrm{d}s,
  \quad \text{a.e.}~0 < t \leqslant T,
\end{equation}
the mild solution to the following fractional diffusion equation
\begin{equation}
  \label{eq:linear}
  (\D_{0+}^\alpha - \mathcal A) w = g, \quad \mbox{with} \;  w(0)=0,
\end{equation}
where the symbol $*$ denotes the convolution.

If $ g = v \delta_0 $ with $ v \in L^2(\Omega) $ and $ \delta_0 $ being the
Dirac measure (in time) concentrated at $ t=0 $, then we call
\begin{equation}
  \label{eq:Sdelta}
  (S(v\delta_0))(t) := E(t) v, \quad 0 < t \leqslant T,
\end{equation}
the mild solution to equation \eqref{eq:linear}. Symmetrically, for any $ g \in
L^1(0,T;L^2(\Omega)) $, we call
\begin{equation} 
  \label{eq:S*g}
  (S^*g)(t) := \int_t^T E^*(s-t) g(s) \, \mathrm{d}s,
  \quad \text{a.e.}~0 < t < T,
\end{equation}
the mild solution to the following backward  fractional diffusion equation:
\begin{equation}
  \label{eq:linear*}
  (\D_{T-}^\alpha - \mathcal A^*) w = g, \quad \mbox{with} \; w (T)=0.
\end{equation}
If $ g = v\delta_T $ with $ v \in L^2(\Omega) $ and $ \delta_T $ being the Dirac
measure (in time) concentrated at $ t=T $, then we call
\begin{equation}
  \label{eq:S*delta}
  (S^*(v\delta_T))(t) := E^*(T-t) v, \quad 0 < t \leqslant T,
\end{equation}
the mild solution to equation \eqref{eq:linear*}. The above $ E^* $ is defined by
\begin{equation}
  \label{eq:E*-def}
  E^*(t) := \frac1{2\pi i} \int_{\Gamma_{\omega_0}}
  e^{tz} R(z^\alpha,\mathcal A^*) \, \mathrm{d}z,
  \quad t > 0.
\end{equation}
Similarly to \eqref{eq:E}, \eqref{eq:E'}, for any $ t>0 $, we have
\begin{align} 
  \nm{E^*(t)}_{\mathcal L(L^2(\Omega))}
  & \leqslant C_{\omega_0,\mathcal M_0} t^{\alpha-1},
  \label{eq:E*} \\
  \nm{(E^*)'(t)}_{\mathcal L(L^2(\Omega))}
  & \leqslant C_{\omega_0,\mathcal M_0} t^{\alpha-2}.
  \label{eq:E*'}
\end{align}
Evidently, for any $ t > 0 $, $ E^*(t) $ is the adjoint operator of $ E(t) $ in
the sense that
\begin{equation}
  \label{eq:E-E*}
  \dual{E(t)v, w}_\Omega = \dual{v, E^*(t)w}_\Omega
  \quad \forall v, w \in L^2(\Omega).
\end{equation}

\begin{remark} 
  By \eqref{eq:E}, a routine calculation (cf.~\cite[Theorem 2.6]{Diethelm2010})
  yields that
  \begin{equation} 
    \label{eq:Sg-C}
    \nm{Sg}_{C([0,T];L^2(\Omega))} \leqslant
    C_{\alpha,q,\omega_0,\mathcal M_0,T} \nm{g}_{L^q(0,T;L^2(\Omega))}
  \end{equation}
  for all $ g \in L^q(0,T;L^2(\Omega)) $ with $ q > 1/\alpha $.
\end{remark}

\begin{remark} 
  For the above solution theory of fractional diffusion equations, we refer the
  reader to \cite{Lubich1996,McLean2010-B,Jin2016}.
\end{remark}

\medskip\noindent{\bf The L1 scheme.} Let $ J \in \mathbb N_{>0} $ and define $
t_j := j\tau $ for each $j=0, 1, 2, \dots, J$, where $ \tau := T/J $. Define $
b_j := j^{1-\alpha}/\Gamma(2-\alpha)$ for each $ j \in \mathbb N $. Assume that
$ g \in L^1(0,T;H^{-1}(\Omega)) $. Applying the L1 scheme \cite{Lin2007} to
problem \eqref{eq:linear} yields the following discretization: seek $
\{W_j\}_{j=1}^J \subset H_0^1(\Omega) $ such that, for any $ 1 \leqslant k
\leqslant J $,
\begin{equation}
  \label{eq:L1}
  b_1 W_k + \sum_{j=1}^{k-1} (b_{k-j+1}-2b_{k-j}+b_{k-j-1})
  W_j - \tau^\alpha \mathcal A W_k =
  \tau^{\alpha-1} \int_{t_{k-1}}^{t_k} g(t) \, \mathrm{d}t
\end{equation}
in $ H^{-1}(\Omega) $, where $ W_{j} $, $ 1 \leqslant j \leqslant J $, is an
approximation of $ w(t_{j}) $. Symmetrically, applying the L1 scheme to problem
\eqref{eq:linear*} yields the following discretization: seek $ \{\mathcal
W_j\}_{j=1}^J \subset H_0^1(\Omega) $ such that, for any $ 1 \leqslant k
\leqslant J $,
\begin{equation}
  \label{eq:L1*}
  b_1 \mathcal W_k + \sum_{j=k+1}^J ( b_{j-k+1} - 2b_{j-k} + b_{j-k-1} )
  \mathcal W_j - \tau^\alpha \mathcal A^* \mathcal W_k =
  \tau^{\alpha-1} \int_{t_{k-1}}^{t_k} g(t) \, \mathrm{d}t
\end{equation}
in $ H^{-1}(\Omega) $. For each $ 1 \leqslant j \leqslant J $, we will use $
S_{\tau,j} g $ and $ S_{\tau,j}^* g $ to denote the above $ W_j $ and $ \mathcal
W_j $, respectively, that is
\begin{equation} \label{eq:Staug}
S_{\tau,j} g := W_j, \quad      S_{\tau,j}^* g := \mathcal W_j.
\end{equation}

In addition, for each $ 1 \leqslant j \leqslant J $, we
define
\begin{equation}
  \label{eq:Stau-delta}
  \mathcal S_{\tau,j}(v\delta_0) := \mathcal
  S_{\tau,j}(v\widehat\delta_0), \quad
  \mathcal S_{\tau,j}^*(v\delta_T) := \mathcal
  S_{\tau,j}^*(v\widehat\delta_T), \quad
\end{equation}
where $ v \in H^{-1}(\Omega) $ and
\begin{align}
  \widehat\delta_0(t) &:=
  \begin{cases}
    \tau^{-1} & \text{ if } 0 < t < t_1, \\
    0 & \text{ if } t_1 < t < T,
  \end{cases} \label{eq:delta_0} \\
  \widehat\delta_T(t) &:=
  \begin{cases}
    0 & \text{ if } 0 < t < t_{J-1}, \\
    \tau^{-1} & \text{ if } t_{J-1} < t < T.
  \end{cases}
  \label{eq:delta_T}
\end{align}

\section{Convergence of the L1 scheme}
\label{sec:L1}
\begin{theorem} 
  \label{thm:conv-Stau}
  Let $ 0< \alpha <1$.  Let $Sg$ and $S_{\tau, j} g$ be defined by \eqref{eq:Sg-l1}
  and \eqref{eq:Staug}, respectively.  Then we have the following estimates:
  \begin{enumerate}
    \item
      For any $ g \in L^\infty(0,T;L^2(\Omega)) $,
      \begin{small}
      \begin{equation}
        \label{eq:S-Stau-g}
        \max_{1 \leqslant j \leqslant J}
        \nm{(Sg)(t_j)-S_{\tau,j}g}_{L^2(\Omega)}
        \leqslant C_{\omega_0,\mathcal M_0}
        \tau^\alpha \Big(
          \frac1\alpha + \frac{1-J^{\alpha-1}}{1-\alpha}
        \Big) \nm{g}_{L^\infty(0,T;L^2(\Omega))}.
      \end{equation}
      \end{small}
    \item
      For any $ v \in L^2(\Omega) $,
      \begin{small}
      \begin{align} 
        \max_{1 \leqslant j \leqslant J} j^{2-\alpha} \nm{
          S(v\delta_0)(t_j) - S_{\tau,j}(v\delta_0)
        }_{L^2(\Omega)} & \leqslant
        C_{\omega_0,\mathcal M_0} \tau^{\alpha-1}
        \nm{v}_{L^2(\Omega)},
        \label{eq:S-Stau-vdelta} \\
        \sum_{j=1}^J \nm{
          S(v\delta_0) - S_{\tau,j}(v\delta_0)
        }_{L^1(t_{j-1},t_j;L^2(\Omega))} & \leqslant
        C_{\omega_0,\mathcal M_0} \tau^\alpha
        \Big(
          \frac1\alpha + \frac{1-J^{\alpha-1}}{1-\alpha}
        \Big) \nm{v}_{L^2(\Omega)}.
        \label{eq:S-Stau-vdelta-l1}
      \end{align}
      \end{small}
  \end{enumerate}
\end{theorem}

\begin{remark} 
  Assume that $ g \in L^\infty(0,T;L^2(\Omega)) $. Passing to the limit $ \alpha
  \to {1-} $ in \eqref{eq:L1} and \eqref{eq:S-Stau-g} yields that, for the parabolic equation
  \[
    w' - \mathcal A w = g, \quad\text{with $ w(0) = 0 $},
  \]
  and the corresponding backward Euler scheme
  \[
    \begin{cases}
      W_0 = 0, \\
      W_k - W_{k-1} - \tau \mathcal A W_k =
      \int_{t_{k-1}}^{t_k} g(t) \, \mathrm{d}t,
      \quad 1 \leqslant k \leqslant J,
    \end{cases}
  \]
  one has the error estimate, noting that $\lim_{\alpha \to 1} \frac{1-J^{\alpha-1}}{1-\alpha} = \ln J$,
  \[
    \max_{1 \leqslant j \leqslant J}
    \nm{w(t_j)-W_j}_{L^2(\Omega)}
    \leqslant C_{\omega_0,\mathcal M_0}
    (1 + \ln J) \tau \nm{g}_{L^\infty(0,T;L^2(\Omega))}.
  \]

\end{remark}

\begin{remark}
  Let us consider the following time fractional diffusion equation
  \[
    \D_{0+}^\alpha(y - y_0)(t) - \mathcal Ay(t) = 0, \quad
    0 < t \leqslant T, \quad
    \text{with } y(0) = y_0,
  \]
  where $ y_0 \in L^2(\Omega) $ is given. Applying the L1 scheme to this
  equation yields the following discretization: seek $ \{W_j\}_{j=1}^J \subset
  H_0^1(\Omega) $ such that, for any $ 1 \leqslant k \leqslant J $,
  \begin{equation*}
    b_1 W_k + \sum_{j=1}^{k-1} (b_{k-j+1}-2b_{k-j}+b_{k-j-1})
    W_j - \tau^\alpha \mathcal A W_k =
    \tau^{\alpha-1} (b_k - b_{k-1}) y_0
  \end{equation*}
  in $ H^{-1}(\Omega) $. Following the proof of \cite[Theorem 3.1]{Jin2016-L1},
  we can use the technical results in Subsection 3.1 to derive that, for any $ 1
  \leqslant j \leqslant J $,
  \[
    \nm{y(t_j) - W_j}_{L^2(\Omega)} \leqslant
    C_{\omega_0,\mathcal M_0} \tau t_j^{-1} \nm{y_0}_{L^2(\Omega)}.
  \]
\end{remark}

The main task of the rest of this section is to prove the above theorem.

\subsection{Some technical results}
\label{ssec:foo}
Define the discrete Laplace transform of $ \{b_j\}_{j=1}^\infty $ by that
\[
  \widehat b(z) := \sum_{j=1}^\infty b_j e^{-jz},
  \quad z \in \Sigma_{\pi/2}.
\]
By the analytic continuation technique, $ \widehat b $ has an analytic
continuation (cf.~\cite[Equation (21)]{Mclean2015Time})
\begin{equation}
  \label{eq:wtb-int}
  \widehat b(z) = \frac1{2\pi i}
  \int_{-\infty}^{(0+)} \frac{e^{w-z}}{1-e^{w-z}}
  w^{\alpha-2} \, \mathrm{d}w,
  \quad z \in \Sigma_{\pi},
\end{equation}
where $ \int_{-\infty}^{({0+})} $ means an integral on a piecewise smooth and
non-self-intersecting path enclosing the negative real axis and orienting
counterclockwise, and $ 0 $ and $ \{z+2k\pi i \neq 0: k \in \mathbb Z\} $ lie on
the different sides of this path. Define
\begin{equation}
  \label{eq:psi-def}
  \psi(z) := (e^z-1)^2 \, \widehat b(z),
  \quad z \in \Sigma_\pi.
\end{equation}
For $ z = x + iy \in \mathbb C \setminus (-\infty,0] $, we have that
(cf.~\cite[Equation (3.7)]{Jin2016-L1})
\begin{small}
\begin{equation}
  \label{eq:re-psi}
  \Re \big( e^{-z} \psi(z) \big) =
  \frac{\sin(\pi(1\!-\!\alpha))}\pi \!
  \int_0^\infty \! \frac{
    s^{\alpha-2}(1\!-\!e^{-s})
    (1\!+\!e^{-2x-s} \!-\! e^{-x-s}\cos y \!-\! e^{-x}\cos y)
  }{
    1-2e^{-x-s}\cos y + e^{-2x-2s}
  } \mathrm{d}s.
\end{equation}
\end{small}

\begin{lemma}
  \label{lem:wtb}
  For any $ L > 0 $, we have
  \begin{equation}
    \label{eq:wtb}
    \sup_{0 < \alpha < 1} \quad \sup_{
      \substack{
        z \in \Sigma_\pi \\
        -L \leqslant \Re z \leqslant 0 \\
        -\pi \leqslant \Im z \leqslant \pi
      }
    } \Snm{\widehat b(z) - z^{\alpha-2}} = C_L.
  \end{equation}
\end{lemma}
\begin{proof} 
  For any $ z \in \Sigma_\pi $ satisfying $ -L \leqslant \Re z \leqslant 0 $ and
  $ 0 \leqslant \Im z \leqslant \pi $, by \eqref{eq:wtb-int}, Cauchy's integral
  theorem and the residue theorem we obtain
  \begin{small}
  \begin{align*}
    \widehat b(z) & = z^{\alpha-2} + \frac1{2\pi i}
    \int_{-\infty -i\pi}^{1-i\pi}
    \frac{e^{w-z}}{1-e^{w-z}} w^{\alpha-2} \, \mathrm{d}w +
    \frac1{2\pi i} \int_{1-i\pi}^{1+i3\pi/2}
    \frac{e^{w-z}}{1-e^{w-z}} w^{\alpha-2} \, \mathrm{d}w \\
    & \qquad {} + \frac1{2\pi i}\int_{1+i3\pi/2}^{-\infty + i3\pi/2}
    \frac{e^{w-z}}{1-e^{w-z}} w^{\alpha-2} \, \mathrm{d}w \\
    & =: z^{\alpha-2} + G(\alpha,z).
  \end{align*}
  \end{small}
  A routine calculation verifies that $ G $ is continuous on
  \[
    [0,1] \times \{
      \xi \in \mathbb C:\
      -L \leqslant \Re \xi \leqslant 0,
      0 \leqslant \Im \xi \leqslant \pi
    \},
  \]
  and so
  \[
    \sup_{0 < \alpha < 1} \quad \sup_{
      \substack{
        -L \leqslant \Re z \leqslant 0 \\
        0 \leqslant \Im z \leqslant \pi
      }
    } \snm{G(\alpha,z)} = C_L.
  \]
  It follows that
  \[ 
    \sup_{0 < \alpha < 1} \sup_{
      \substack{
        z \in \Sigma_\pi \\
        -L \leqslant \Re z \leqslant 0 \\
        0 \leqslant \Im z \leqslant \pi
      }
    } \Snm{\widehat b(z) - z^{\alpha-2}} = C_L.
  \]
  Similarly,
  \[
    \sup_{0 < \alpha < 1} \sup_{
      \substack{
        z \in \Sigma_\pi \\
        -L \leqslant \Re z \leqslant 0 \\
        -\pi \leqslant \Im z \leqslant 0
      }
    } \Snm{\widehat b(z) - z^{\alpha-2}} = C_L.
  \]
  Combining the above two estimates proves \eqref{eq:wtb} and hence this lemma.
\end{proof}

\begin{lemma}
  \label{lem:psi}
  For any $ 0 < \delta < \pi $ and $ L > 0 $, we have
  \begin{align}
    \inf_{0 < \alpha < 1} \quad \inf_{
      \delta \leqslant y \leqslant \pi
    } \Re \big( e^{-iy}\psi(iy) \big) & = C_\delta,
    \label{eq:psi-1} \\
    \sup_{0 < \alpha < 1} \quad \sup_{
      \substack{
        -L \leqslant \Re z \leqslant 0 \\
        \delta \leqslant \Im z  \leqslant \pi
      }
    }\Snm{
      \frac{\mathrm{d}}{\mathrm{d}z} (e^{-z}\psi(z))
    } & = C_{\delta,L}.
    \label{eq:psi-2}
  \end{align}
\end{lemma}
\begin{proof}
  For any $ \delta \leqslant y \leqslant \pi $, we have, by \eqref{eq:re-psi}
  with $ z= 0+i y $,
  \begin{align*}
    & \Re \big ( e^{-iy} \psi (iy) \big )
    = \frac{\sin (\pi (1- \alpha))}{\pi} \int_{0}^{\infty}
    \frac{
      s^{\alpha-2} (1- e^{-2s}) (1- \cos y)
    }{1- 2 e^{-s} \cos y + e^{-2 s}} \, \mathrm{d}s  \\
    & >  \frac{\sin(\pi(1-\alpha))}\pi (1-\cos\delta)
    \int_0^\infty \frac{s^{\alpha-2}(1-e^{-2s})}{
      1+2e^{-s}+e^{-2s}
    } \, \mathrm{d}s, \\
    &=  \frac{\sin(\pi(1-\alpha))}\pi (1-\cos\delta)
    \Big [
      \int_0^1 \frac{s^{\alpha-2}(1-e^{-2s})}{
        1+2e^{-s}+e^{-2s}
      } \, \mathrm{d}s
      +
      \int_1^\infty \frac{s^{\alpha-2}(1-e^{-2s})}{
        1+2e^{-s}+e^{-2s}
      } \, \mathrm{d}s
      \Big ].
    \end{align*}
    In view of the two simple estimates
    \begin{align*}
      \int_0^1 \frac{s^{\alpha-2}(1-e^{-2s})}{
        1+2e^{-s}+e^{-2s}
      } \, \mathrm{d}s
      & > \int_0^1 \frac{s^{\alpha-2}(e^{-2s} 2s)}{4 } \, \mathrm{d}s
      = \int_0^1 \frac{s^{\alpha-1}(e^{-2s})}{2 } \, \mathrm{d}s \\
      & > \int_0^1 \frac{s^{\alpha-1}(e^{-2})}{2
    } \, \mathrm{d}s
    = \frac{e^{-2}}{2 \alpha}
  \end{align*}
  and
  \begin{align*}
    &\int_1^\infty \frac{s^{\alpha-2}(1-e^{-2s})}{
      1+2e^{-s}+e^{-2s}
    } \, \mathrm{d}s
    > \int_1^{\infty} s^{\alpha-2} \frac{1-e^{-2}}{4} \, \mathrm{d}s =
    \frac{1-e^{-2}}{4(1-\alpha)},
  \end{align*}
  we then obtain, for any $ \delta \leqslant y \leqslant \pi $,
  \begin{align*}
    \Re \big( e^{-iy} \psi(iy) \big)
    \geqslant \frac{\sin(\pi(1-\alpha))}\pi
    (1-\cos\delta) \Big(
      \frac{e^{-2}}{2\alpha} + \frac{1-e^{-2}}{4(1-\alpha))}
    \Big) \geqslant C_\delta.
  \end{align*}
  This implies inequality \eqref{eq:psi-1}.

  Now let us prove \eqref{eq:psi-2}. For any $ z \in \mathbb C $ satisfying $
  \delta \leqslant \Im z \leqslant \pi $, using the residue theorem yields, by
  \eqref{eq:wtb-int}, that
  \begin{equation}
    \label{eq:wtb-series}
    \widehat b(z) = \sum_{k = -\infty}^\infty
    (z+2k\pi i)^{\alpha-2},
  \end{equation}
  and hence
  \[
    \widehat b'(z) = (\alpha-2)\sum_{k = -\infty}^\infty
    (z+2k\pi i)^{\alpha-3}.
  \]
  A simple calculation then gives
  \[
    \sup_{0 < \alpha < 1} \quad \sup_{
      \substack{
        -L \leqslant \Re z \leqslant 0 \\
        \delta \leqslant \Im z \leqslant \pi
      }
    } \snm{e^{-z}(e^z-1)^2 \widehat b'(z)} = C_{\delta,L}.
  \]
  In addition, Lemma \ref{lem:wtb} implies
  \[
    \sup_{0 < \alpha < 1} \quad \sup_{
      \substack{
        -L \leqslant \Re z \leqslant 0 \\
        \delta \leqslant \Im z \leqslant \pi
      }
    }\snm{(e^z-e^{-z})\widehat b(z)} = C_{\delta,L}.
  \]
  Consequently, \eqref{eq:psi-2} follows from the equality
  \[
    \frac{\mathrm{d}}{\mathrm{d}z} (e^{-z}\psi(z)) =
    (e^z-e^{-z})\widehat b(z) +
    e^{-z}(e^z-1)^2 \widehat b'(z),
    \quad z \in \Sigma_\pi.
  \]
  This completes the proof.
\end{proof}

\begin{lemma} 
  \label{lem:psi-esti}
  Assume that $ \pi/2 < \theta_0 < \pi $. Then there exists $ \pi/2 < \theta^*
  \leqslant \theta_0 $ depending only on $ \theta_0 $ such that
  \begin{equation}
    \label{eq:psi_sector}
    e^{-z} \psi(z) \in \Sigma_{\theta_0}
    \quad\text{ for all } z \in \Sigma_{\theta^*}
    \text{ with } -\pi \leqslant \Im z \leqslant \pi
  \end{equation}
  and
  \begin{equation}
    \label{eq:psi-esti}
    \snm{e^{-z}\psi(z)} \geqslant C_{\theta_0}
    \snm{z}^\alpha \quad \text{for all }
    z \in \Upsilon_{\theta^*} \setminus \{0\}.
  \end{equation}
\end{lemma}
\begin{proof}
  Step 1. By \eqref{eq:re-psi}, a simple calculation gives
  \[
    \Re \big( e^{-z}\psi(z) \big) > 0
    \text{ for all } z \in D_1,
  \]
  so that
  \begin{equation}
    \label{eq:bj-1}
    e^{-z}\psi(z) \in \Sigma_{\pi/2} \quad\text{for all } z \in D_{1},
  \end{equation}
  where
  \[
    D_{1}= \{
      z \in \mathbb{C}: \Re z \geqslant 0, \, 0 \leqslant \Im z
      \leqslant \pi, \, z \ne 0
    \}.
  \]

  Step 2. From \eqref{eq:psi-def} and \eqref{eq:wtb} we conclude that there
  exists a continuous function $ G $ on $ (0,1) \times D_2 $, such that
  \begin{equation}
    \label{eq:64}
    e^{-z}\psi(z) = z^\alpha\big(1 + z G(\alpha,z)\big)
    \quad \forall z \in D_2
  \end{equation}
  and that
  \[
    \sup_{0 < \alpha <1} \sup_{
      z \in D_2
    } \snm{G(\alpha,z)} = C_{\theta_0},
  \]
  where
  \[
    D_2 := \{
      \xi \in \mathbb C \setminus \{0\}:
      \ \pi/2 \leqslant \mbox{Arg}(\xi) \leqslant \theta_0,\,
      0 < \Im \xi \leqslant \pi
    \}.
  \]
  Hence, there exists $ 0 < \epsilon_0 < \pi $, depending only on $ \theta_0
  $, such that
  \begin{align*}
    & \Snm{
      \mbox{Arg}(1+zG(\alpha,z))
    } \leqslant (\theta_0-\pi/2)/2 \quad \text{ and } \quad
    \snm{e^{-z}\psi(z)} \geqslant C_{\theta_0}
    \snm{z}^\alpha \\
    & \text{ for all } z \in \Sigma_{\theta_0} \setminus
    \Sigma_{\pi/2} \text{ with }
    0 < \Im z \leqslant \epsilon_0.
  \end{align*}
  Since
  \begin{align*}
    & \mbox{Arg} \big ( e^{-z}  \psi (z)  \big ) =
    \mbox{Arg} \big ( z^{\alpha}  ( 1+ zG(\alpha,z)) \big )
    \quad \text{(by \eqref{eq:64})} \\
    ={} &
    \alpha \mbox{Arg}(z) + \mbox{Arg} \big ( 1 + zG(\alpha,z) \big ),
  \end{align*}
  it follows that
  \begin{equation}
    \label{eq:bj-2}
    \begin{aligned}
      & e^{-z}\psi(z) \in \Sigma_{\theta_0} \text{ and }
      \snm{e^{-z}\psi(z)} \geqslant C_{\theta_0} \snm{z}^\alpha \\
      & \text{for all} \,
      z \in \Sigma_{(\theta_0+\pi/2)/2} \setminus \Sigma_{\pi/2} \, \text{ with }
      0 < \Im z \leqslant \epsilon_0.
    \end{aligned}
  \end{equation}

  Step 3. Note that $ \epsilon_0 $ is a constant depending only on $ \theta_0 $.
  By \eqref{eq:psi-1} we have
  \[ 
    \inf_{0 < \alpha < 1} \, \inf_{
      \substack{
        \Re z =0  \\
        \epsilon_0 \leqslant \Im z \leqslant \pi
      }
    } \Re \big ( e^{-z}\psi(z)  \big ) = C_{\theta_0}.
  \]
  From \eqref{eq:psi-2} we then conclude that there exists $ 0 < \epsilon_1 <
  \pi $,
  depending only on $ \theta_0 $, such that
  \[ 
    \inf_{0 < \alpha < 1} \, \inf_{
      \substack{
        -\epsilon_1 \leqslant \Re z \leqslant 0 \\
        \epsilon_0 \leqslant \Im z \leqslant \pi
      }
    } \Re \big ( e^{-z}\psi(z)  \big ) = C_{\theta_0} > 0.
  \]
  It follows that
  \begin{equation}
    \label{eq:bj-31}
    \begin{aligned}
      & e^{-z} \psi(z) \in \Sigma_{\pi/2} \text{ and }
      \snm{e^{-z}\psi(z)} \geqslant C_{\theta_0}  \text{ for all } \\
      & z \in \Sigma_{(\theta_0+\pi/2)/2}
      \setminus \Sigma_{\pi/2} \, \text{ with } -\epsilon_1 \leqslant \Re z
      \leqslant 0 \text{ and }
      \epsilon_0 \leqslant \Im z \leqslant \pi.
    \end{aligned}
  \end{equation}
  Letting $ \theta^* := \pi/2 + \arctan(\epsilon_1/\pi) $, by
  \eqref{eq:bj-1}, \eqref{eq:bj-2} and \eqref{eq:bj-31} we obtain that
  \begin{equation}
    \label{eq:zq-1}
    e^{-z} \psi(z) \in \Sigma_{\theta_0} \quad
    \text{for all } z \in \Sigma_{\theta^*}
    \text{ with } 0 \leqslant \Im z \leqslant \pi
  \end{equation}
  and that
  \begin{equation}
    \label{eq:zq-2}
    \snm{e^{-z}\psi(z)} \geqslant C_{\theta_0}
    \snm{z}^\alpha \text{ for all } z \in \Upsilon_{\theta^*}
    \,\text{ with }\, 0 < \Im z \leqslant \pi.
  \end{equation}


  Step 4. By the fact that
  \[
    \overline{e^{-z}\psi(z)} =
    e^{-\overline{z}} \psi(\overline z)
    \quad\text{ for all } z \in \Sigma_\pi,
  \]
  using \eqref{eq:zq-1} and \eqref{eq:zq-2} proves \eqref{eq:psi_sector} and
  \eqref{eq:psi-esti}, respectively. This completes the proof.
\end{proof}

By \eqref{eq:psi-def} and Lemma \ref{lem:wtb}, a routine calculation gives the following lemma.
\begin{lemma} 
  Assume that $ \pi/2 < \theta < \pi $. Then
  \begin{equation} 
    \label{eq:129}
    \snm{\psi(z) - z^\alpha} \leqslant
    C_\theta \snm{z}^{\alpha+1}
  \end{equation}
  for all $ z \in \Upsilon_\theta \setminus \{0\} $.
\end{lemma}

\begin{remark}
In   Lemma \ref{lem:psi-esti}, we prove that for any given  $\theta_{0} \in (\pi/2, \pi)$, we can show that  $e^{-z} \psi (z) \in \Sigma_{\theta_{0}}$ for $z \in \Sigma_{\theta^{*}}$ with some $\pi/2 < \theta^{*} \leqslant \theta_{0}$. Therefore our error estimates hold  for any elliptic operator $\mathcal{A}$ where  the resolvent set of $\mathcal{A}$  lies in $\Sigma_{\theta_{0}}$. The techniques used in the proof of Lemma \ref{lem:psi-esti} are new and may be extended to consider the error estimates for the higher order L-type schemes.   Let us recall some available approach in literature for proving  Lemma \ref{lem:psi-esti}.   In Jin et al. \cite{Jin2016-L1} the authors use the following steps  to show $e^{-z} \psi (z) \in \Sigma_{\theta_{0}}$:

Step 1.  Let $z \in \{ z: \mbox{Arg} (z) = \theta^{*}= \pi/2 \}$ and prove that  $e^{-z} \psi (z) \in \Sigma_{\theta_{0}}$ for some suitable $\theta_{0} \in (\pi/2, \pi)$.

Step 2. By the continuity of  $e^{-z} \psi (z)$ with respect to $\theta^{*}$, one may claim  that $e^{-z} \psi (z) \in \Sigma_{\theta_{0}}$ also for $\theta^{*} \in (\pi/2, \pi)$ for $\theta^{*}$ sufficiently  close to $\pi/2$.

By using this approach, Jin et al. \cite{Jin2016-L1} show that $\theta_{0} = 3 \pi/4 - \epsilon$, with $\epsilon >0$,   which implies that this approach  do not work for the elliptic operator $\mathcal{A}$ where the resolvent set of $\mathcal{A}$  lies in $\Sigma_{\theta_{0}}$ with $\theta_{0} < 3 \pi/4$.  It seems also very difficult to prove the similar results as in Lemma \ref{lem:psi-esti} for the higher order L-type scheme by using the approach in \cite{Jin2016-L1}. Therefore the new techniques developed in the proof of  Lemma \ref{lem:psi-esti} may open a door to consider the numerical analysis for high order L-type schemes for solving time fractional partial differential equations.
\end{remark}


\subsection{Proof of  Theorem \ref{thm:conv-Stau}}

By Lemma \ref{lem:psi-esti}, there exists $ \pi/2 < \omega^* \leqslant \omega_0 $,
depending only on $ \omega_0 $, such that
\begin{equation}
  \label{eq:0}
  e^{-z} \psi(z) \in \Sigma_{\omega_0}
  \text{
    for all $ z \in \Sigma_{\omega^*} $
    with $ -\pi \leqslant \operatorname{Im} z \leqslant \pi $
  }
\end{equation}
and that
\begin{equation}
  \label{eq:psi>}
  \snm{e^{-z}\psi(z)} \geqslant C_{\omega_0}
  \snm{z}^\alpha \quad \text{for all }
  z \in \Upsilon_{\omega^*} \setminus \{0\}.
\end{equation}
Define
\begin{equation}
  \label{eq:calE-def}
  \mathcal E(t) := \tau^{-1} \mathcal E_{\lfloor t/\tau \rfloor},
  \quad t > 0,
\end{equation}
where $ \lfloor \cdot \rfloor $ is the floor function and
\begin{equation}
  \label{eq:calEj}
  \mathcal E_j := \frac1{2\pi i} \int_{\Upsilon_{\omega^*}}
  e^{jz} R(\tau^{-\alpha}e^{-z}\psi(z), \mathcal A) \, \mathrm{d}z,
  \quad j \in \mathbb N.
\end{equation}
Note that (\ref{eq:rho(A)}) and \eqref{eq:0} guarantee that the above $ \mathcal
E_j $ is well defined, and we recall that $ \psi $ is defined by
\eqref{eq:psi-def}.

\begin{lemma} 
  For any $ g \in L^1(0,T;L^2(\Omega)) $, we have
  \begin{equation}
    \label{eq:Stau-g}
    S_{\tau,j}g = \int_0^{t_j} \mathcal E(t_j-t) g(t) \, \mathrm{d}t
    \quad \forall 1 \leqslant j \leqslant J.
  \end{equation}
\end{lemma}
\begin{proof} 
  Since the techniques used in this proof are standard in the theory of Laplace
  transform, we only provide a brief proof; see
  \cite{Mclean2015Time,Jin2015IMA,Yan2018} for more details. Extend $ g $ to $
  (T,\infty) $ by zero and define $ t_j := j\tau $ for each $ j > J $. Define $
  \{W_k\}_{k=1}^\infty \subset H_0^1(\Omega) $ by that, for any $ k \geqslant 1
  $,
  \begin{equation}
    \label{eq:W}
    b_1 W_k + \sum_{j=1}^{k-1} (b_{k-j+1}-2b_{k-j}+b_{k-j-1}) W_j
    - \tau^\alpha \mathcal A W_k =
    \tau^{\alpha-1} \int_{t_{k-1}}^{t_k} g(t) \, \mathrm{d}t
  \end{equation}
  in $ H^{-1}(\Omega) $. By definition,
  \begin{equation}
    \label{eq:Staujg}
    S_{\tau,j} g = W_j, \quad \forall 1 \leqslant j \leqslant J.
  \end{equation}
  The rest of this proof is divided into three steps.

  Step 1. We prove that the following discrete Laplace transform of $
  \{W_k\}_{k=1}^\infty $ is analytic on $ \Sigma_{\pi/2} $:
  \begin{equation}
    \label{eq:wtW}
    \widehat W(z) := \sum_{k=1}^\infty e^{-kz} W_k,
    \quad z \in \Sigma_{\pi/2}.
  \end{equation}
  Note first that we can assume that $ g \in L^\infty(0,\infty;L^2(\Omega)) $.
  Since
  \[
    \sup_{a > 0} \, \nm{g}_{{}_0H^{-\alpha/2}(0,a;L^2(\Omega))}
    < \infty, 
  \]
  by the techniques to prove \eqref{eq:Stau-Stauj} and \eqref{eq:Stau-stab-infty}
  we can obtain
  \[
    \sup_{k \geqslant 1} \, \nm{W_k}_{L^2(\Omega)} < \infty.
  \]
  Therefore, it is evident that $ \widehat W $ is analytic on $ \Sigma_{\pi/2}
  $.

  Step 2. Let us prove that, for any $ 1 \leqslant j \leqslant J $,
  \begin{equation}
    \label{eq:eve}
    W_j = \sum_{k=1}^J \frac{\tau^{-1}}{2\pi i}
    \int_{1-\pi i}^{1+\pi i} R(\tau^{-\alpha}e^{-z}\psi(z),\mathcal A)
    e^{(j-k)z} \, \mathrm{d}z
    \int_{t_{k-1}}^{t_k} g(t) \, \mathrm{d}t.
  \end{equation}
  Multiplying both sides of \eqref{eq:W} by $ e^{-kz} $ and summing over $ k $
  from $1$ to $\infty$, we obtain
  \[
    \big(
      (e^z-2+e^{-z})\widehat b(z) - \tau^\alpha \mathcal A
    \big) \widehat W(z) =
    \tau^{\alpha-1} \sum_{k=1}^\infty
    \int_{t_{k-1}}^{t_k} g(t) \, \mathrm{d}t
    e^{-kz}, \quad \forall z \in \Sigma_{\pi/2},
  \]
  which, together with \eqref{eq:psi-def}, yields
  \begin{equation}
    \label{eq:lxy}
    (e^{-z}\psi(z) - \tau^\alpha \mathcal A) \widehat W(z) =
    \tau^{\alpha-1} \sum_{k=1}^\infty
    \int_{t_{k-1}}^{t_k} g(t) \, \mathrm{d}t
    e^{-kz}, \quad \forall z \in \Sigma_{\pi/2}.
  \end{equation}
  Hence, from (\ref{eq:rho(A)}), \eqref{eq:0} and the fact $ g|_{(T,\infty)} = 0
  $, it follows that
  \begin{align*}
    \widehat W(z) &= \tau^{-1} R(\tau^{-\alpha} e^{-z}\psi(z),\mathcal A)
    \sum_{k=1}^\infty \int_{t_{k-1}}^{t_k}
    g(t) \, \mathrm{d}t e^{-kz} \\
    & = \tau^{-1} R(\tau^{-\alpha} e^{-z}\psi(z),\mathcal A)
    \sum_{k=1}^J \int_{t_{k-1}}^{t_k} g(t) \, \mathrm{d}t e^{-kz}
  \end{align*}
  for all $ z \in \Sigma_{\pi/2} $ with $ -\pi \leqslant \operatorname{Im} z
  \leqslant \pi $. Therefore, \eqref{eq:eve} follows from the equality
  \[
    W_j  = \frac1{2\pi i} \int_{1-\pi i}^{1+\pi i}
    \widehat W(z) e^{jz} \, \mathrm{d}z,
  \]
  which is evident by \eqref{eq:wtW}.

  Step 3. By Cauchy's integral theorem, we have, for any $a>1$, when $ k \geqslant
  j+1$,
  \begin{align}
    &\qquad   \Big \| \int_{1-\pi i}^{1+\pi i}
    R(\tau^{-\alpha}e^{-z}\psi(z),\mathcal A)
    e^{(j-k)z} \, \mathrm{d}z  \Big \|_{\mathcal L(L^2(\Omega))} \notag \\
    &   =\Big \|  \int_{a-\pi i}^{a+\pi i}
    R(\tau^{-\alpha}e^{-z}\psi(z),\mathcal A)
    e^{(j-k)z} \, \mathrm{d}z  \Big \|_{\mathcal L(L^2(\Omega))} \notag \\
    & \leqslant \mathcal M_0 e^{(j-k) a}
    \int_{a-\pi i}^{a+\pi i} \frac{|dz|}{\tau^{-\alpha} | e^{-z} \psi (z) |}
    \quad\text{(by \eqref{eq:R(z,A)}).}
    \label{eq:zq}
  \end{align}
  Since \eqref{eq:re-psi} implies
  \[
    \snm{e^{-z}\psi(z)} \geqslant C_\alpha
    \quad\text{ for all $ z \in \mathbb C $ with $ \Re z \geqslant 1 $},
  \]
  passing to the limit $ a \to \infty $ in \eqref{eq:zq} yields
  \[
    \int_{1-\pi i}^{1+\pi i}
    R(\tau^{-\alpha}e^{-z}\psi(z),\mathcal A)
    e^{(j-k)z} \, \mathrm{d}z = 0,  \quad \mbox{for} \; k \geqslant j+1.
  \]
  Thus from \eqref{eq:eve} we obtain
  \begin{align*}
    W_j &= \sum_{k=1}^j \frac{\tau^{-1}}{2\pi i}
    \int_{1-\pi i}^{1+\pi i} R(\tau^{-\alpha}e^{-z}\psi(z),\mathcal A)
    e^{(j-k)z} \, \mathrm{d}z
    \int_{t_{k-1}}^{t_k} g(t) \, \mathrm{d}t \\
    &= \sum_{k=1}^j \mathcal E_{j-k} \int_{t_{k-1}}^{t_k}
    g(t) \, \mathrm{d}t =
    \int_0^{t_j} \mathcal E(t_j-t) g(t) \, \mathrm{d}t.
  \end{align*}
  Here we have used the equality
  \[
    \int_{1-\pi i}^{1+\pi i}
    R(\tau^{-\alpha} e^{-z} \psi(z), \mathcal A) e^{(j-k)z} \, \mathrm{d}z
    = \int_{\Upsilon{\omega^*}}
    R(\tau^{-\alpha} e^{-z} \psi(z), \mathcal A) e^{(j-k)z} \, \mathrm{d}z,
  \]
  which can be easily verified by Cauchy's integral theorem. By
  \eqref{eq:Staujg}, this proves \eqref{eq:Stau-g} and thus completes the proof.
\end{proof}

\begin{remark}
  In \eqref{eq:Stau-g}, we use the piecewise kernel function $\mathcal{E}(t)$ to
  express the discrete solution $S_{\tau, j}g$, which is different from the
  discrete solution expression in literature \cite{Jin2016-L1,Yan2018}, where
  the authors assumed that the function $g$ has more regularities at 0 and has
  the Taylor expansion at $0$ and then applied the convolution techniques for
  obtaining the discrete solution. In our paper, we only assume that $ g \in
  L^\infty(0,T;L^2(\Omega))$ and we did not use the convolution techniques for
  obtaining the discrete solutions as in \cite{Jin2016-L1,Yan2018}. One may use
  the similar idea to consider more general function $g$; for example, $g$ is a
  stochastic Wiener process $g= \frac{d W(t)}{dt}$, where $W$ is the Hilbert
  space valued cylindrical Wiener process.
\end{remark}

\begin{lemma}
  \label{lem:resolvent1}
  For any $ z \in \Upsilon_{\omega^*} \setminus \{0\} $,
  \begin{equation}
    \label{eq:lsj-1}
    \nm{
      e^{z} R(\tau^{-\alpha}z^\alpha,\mathcal A) -
      R(\tau^{-\alpha}e^{-z}\psi(z),\mathcal A)
    }_{\mathcal L(L^2(\Omega))}
    \leqslant C_{\omega_0,\mathcal M_0}
    \snm{z}^{1-\alpha} \tau^{\alpha}.
  \end{equation}
\end{lemma}
\begin{proof} 
  We have
  \begin{align*}
    & e^{z}R(\tau^{-\alpha}z^\alpha,\mathcal A) -
    R(\tau^{-\alpha}e^{-z}\psi(z),\mathcal A) \\
    ={} &
    \big(
      \tau^{-\alpha} \big( \psi(z) - z^\alpha) + (1-e^{z})\mathcal A
    \big) R(\tau^{-\alpha} z^\alpha, \mathcal A)
    R(\tau^{-\alpha}e^{-z}\psi(z),\mathcal A) \\
    ={} &
    \mathbb I_1 + \mathbb I_2,
  \end{align*}
  where
  \begin{align*}
    \mathbb I_1 &:= \tau^{-\alpha}(\psi(z)-z^\alpha)
    R(\tau^{-\alpha}z^\alpha,\mathcal A)
    R(\tau^{-\alpha} e^{-z}\psi(z),\mathcal A), \\
    \mathbb I_2 &:= (1-e^z)\mathcal A R(\tau^{-\alpha}z^\alpha,\mathcal A)
    R(\tau^{-\alpha} e^{-z}\psi(z),\mathcal A).
  \end{align*}
  Note that \eqref{eq:R(z,A)}, \eqref{eq:0} and \eqref{eq:psi>} imply
  \begin{align}
    \nm{R(\tau^{-\alpha}z^\alpha, \mathcal A)}_{\mathcal L(L^2(\Omega))}
    & \leqslant C_{\mathcal M_0}
    \snm{z}^{-\alpha} \tau^{\alpha}, \label{eq:lxy-1} \\
    \nm{
      R(\tau^{-\alpha}e^{-z}\psi(z), \mathcal A)
    }_{\mathcal L(L^2(\Omega))}
    & \leqslant C_{\omega_0,\mathcal M_0}
    \snm{z}^{-\alpha} \tau^{\alpha}. \label{eq:lxy-2}
  \end{align}
  By \eqref{eq:129}, \eqref{eq:lxy-1} and \eqref{eq:lxy-2} we have
  \begin{align*}
    \nm{\mathbb I_1}_{\mathcal L(L^2(\Omega))}
    & \leqslant C_{\omega_0,\mathcal M_0}
    \snm{z}^{1-\alpha} \tau^\alpha.
  \end{align*}
  Since
  \begin{align*} 
    & \nm{
      \mathcal AR(\tau^{-\alpha}z^\alpha,\mathcal A)
      R(\tau^{-\alpha}e^{-z}\psi(z), \mathcal A)
    }_{\mathcal L(L^2(\Omega))} \\
    ={} &
    \nm{
      (\tau^{-\alpha}z^\alpha R(\tau^{-\alpha}z^\alpha,\mathcal A) - I)
      R(\tau^{-\alpha}e^{-z}\psi(z), \mathcal A)
    }_{\mathcal L(L^2(\Omega))} \\
    \leqslant{} &
    C_{\omega_0,\mathcal M_0}
    \snm{z}^{-\alpha} \tau^\alpha
    \quad\text{(by \eqref{eq:lxy-1} and \eqref{eq:lxy-2}),}
  \end{align*}
  we obtain
  \[
    \nm{\mathbb I_2}_{\mathcal L(L^2(\Omega))}
    \leqslant C_{\omega_0,\mathcal M_0}
    \snm{z}^{1-\alpha} \tau^\alpha.
  \]
  Combining the above estimates of $ \mathbb I_1 $ and $ \mathbb I_2 $ proves
  \eqref{eq:lsj-1} and hence this lemma.
\end{proof}

\begin{lemma} 
  \label{lem:E-calE}
  For any $ 1 \leqslant j \leqslant J $,
  \begin{equation}
    \label{eq:E-calE}
    \nm{E(t_j) - \mathcal E(t_j-)}_{\mathcal L(L^2(\Omega))}
    \leqslant C_{\omega_0,\mathcal M_0}
    \tau^{\alpha-1} j^{\alpha-2}.
  \end{equation}
\end{lemma}
\begin{proof} 
  Inserting $ t = t_j $ into \eqref{eq:E-def} yields
  \[ 
    E(t_j) = \frac1{2\pi i} \int_{\Gamma_{\omega^*}} e^{t_jz}
    R(z^\alpha,\mathcal A) \, \mathrm{d}z  =
    \frac{\tau^{-1}}{2\pi i}
    \int_{\Gamma_{\omega^*}} e^{jz}
    R(\tau^{-\alpha}z^\alpha, \mathcal A) \, \mathrm{d}z,
  \]
  so that from \eqref{eq:calE-def} and \eqref{eq:calEj} it follows that
  \[
    E(t_j) - \mathcal E(t_j-) = \mathbb I_1 + \mathbb I_2,
  \]
  where
  \begin{align*} 
    \mathbb I_1 &:= \frac{\tau^{-1}}{2\pi i}
    \int_{\Gamma_{\omega^*}\setminus\Upsilon_{\omega^*}}
    e^{jz} R(\tau^{-\alpha}z^\alpha,\mathcal A) \, \mathrm{d}z, \\
    \mathbb I_2 &:= \frac{\tau^{-1}}{2\pi i}
    \int_{\Upsilon_{\omega^*}} e^{(j-1)z} \big(
      e^z R(\tau^{-\alpha}z^\alpha,\mathcal A) -
      R(\tau^{-\alpha}e^{-z}\psi(z),\mathcal A)
    \big) \, \mathrm{d}z.
  \end{align*}
  For $ \mathbb I_1 $, we have, by \eqref{eq:R(z,A)},
  \begin{align*} 
    & \nm{\mathbb I_1}_{\mathcal L(L^2(\Omega))} \leqslant
    C_{\mathcal M_0} \tau^{-1} \int_{\pi/\sin\omega^*}^\infty
    e^{j\cos\omega^* r}  (\tau^{\alpha}r^{-\alpha} )  \, \mathrm{d}r  \\
    \leqslant {}
    & C_{\mathcal M_0} \tau^{\alpha-1}
    \int_{\pi/\sin\omega^*}^\infty
    e^{j\cos\omega^* r} r^{-\alpha} \, \mathrm{d}r \\
    \leqslant {}
    & C_{\mathcal M_0} \tau^{\alpha-1}
    \int_{\pi/\sin\omega^*}^\infty e^{j\cos\omega^* r} r^{1-\alpha}  \, \mathrm{d}r
\quad \mbox{( since}\,  r \, \mbox{is lower bounded)}
  \\
    \leqslant {}  & C_{\omega_0,\mathcal M_0} \tau^{\alpha-1}
    j^{\alpha-2} e^{j\pi\cot\omega^*}
    \leqslant   C_{\omega_0,\mathcal M_0} \tau^{\alpha-1}
    j^{\alpha-2}.
  \end{align*}
  For $ \mathbb I_2 $, by \eqref{eq:lsj-1} we obtain
  \begin{align*} 
    \nm{\mathbb I_2}_{\mathcal L(L^2(\Omega))}
    & \leqslant C_{\omega_0,\mathcal M_0}
    \tau^{-1} \int_0^{\pi/\sin\omega^*}
    e^{(j-1)\cos\omega^* r} r  \big (\tau^{\alpha} r^{-\alpha} \big ) \, \mathrm{d}r  \\
    & \leqslant C_{\omega_0,\mathcal M_0}
    \tau^{\alpha-1} \int_0^{\pi/\sin\omega^*}
    e^{(j-1)\cos\omega^* r} r^{1-\alpha} \, \mathrm{d}r  \\
& \leqslant C_{\omega_0,\mathcal M_0}
    \tau^{\alpha-1} \int_0^{\pi/\sin\omega^*}
    e^{ j\cos\omega^* r} r^{1-\alpha} \, \mathrm{d}r
\leqslant C_{\omega_0,\mathcal M_0}
    \tau^{\alpha-1} j^{\alpha-2}.
  \end{align*}
  Combining the above estimates of $ \mathbb I_1 $ and $ \mathbb I_2 $ yields
  \eqref{eq:E-calE} and thus concludes the proof.
\end{proof}

\begin{lemma}
  \label{lem:int-E-calE}
  We have
  \begin{small}
  \begin{equation}
    \label{eq:int-E-wtE}
    \nm{E-\mathcal E}_{L^1(0,T;\mathcal L(L^2(\Omega)))}
    \leqslant C_{\omega_0,\mathcal M_0} \,
    \Big(
      \frac1\alpha + \frac{1-J^{\alpha-1}}{1-\alpha}
    \Big) \tau^\alpha.
  \end{equation}
  \end{small}
\end{lemma}
\begin{proof} 
  By \eqref{eq:E} we have
  \begin{equation} 
    \label{eq:E-Et1}
    \nm{E - E(t_1)}_{L^1(0,t_1;\mathcal L(L^2(\Omega)))}
    \leqslant C_{\omega_0,\mathcal M_0}
    \tau^\alpha  \alpha^{-1},
  \end{equation}
  and a straightforward calculation gives, by \eqref{eq:E'},
  \begin{align} 
    & \sum_{j=2}^J \nm{
      E - E(t_j)
    }_{L^1(t_{j-1},t_j;\mathcal L(L^2(\Omega)))}
    \leqslant \tau \nm{E'}_{L^1(t_1,T;\mathcal L(L^2(\Omega)))} \notag \\
    \leqslant{} &
    C_{\omega_0,\mathcal M_0} \tau \int_{t_1}^T t^{\alpha-2} \, \mathrm{d}t
    = C_{\omega_0,\mathcal M_0} \tau^\alpha (1-J^{\alpha-1})(1-\alpha)^{-1}.
    \label{eq:E-Ej}
  \end{align}
  It follows that
  \begin{align*}
    & \sum_{j=1}^J \nm{E-E(t_j)}_{
      L^1(t_{j-1},t_j;\mathcal L(L^2(\Omega)))
    }
    \leqslant
    C_{\omega_0,\mathcal M_0} \tau^\alpha \Big(
      \alpha^{-1} + (1-J^{\alpha-1})(1-\alpha)^{-1}
    \Big).
  \end{align*}
  Further we have, by  Lemma \ref{lem:E-calE},
  \begin{align*} 
    \sum_{j=1}^J \tau \nm{E(t_j) - \mathcal E(t_j-)}_{
      \mathcal L(L^2(\Omega))
    } & \leqslant C_{\omega_0,\mathcal M_0}
    \tau^\alpha \sum_{j=1}^J j^{\alpha-2} \\
    & \leqslant C_{\omega_0,\mathcal M_0}
    \tau^\alpha (1-J^{\alpha-1})(1-\alpha)^{-1}.
  \end{align*}
  Thus we get
  \begin{align*}
    & \nm{E-\mathcal E}_{L^1(0,T;\mathcal L(L^2(\Omega)))} \\
    \leqslant{} & \sum_{j=1}^J \Big(
      \nm{E-E(t_j)}_{L^1(t_{j-1},t_j;\mathcal L(L^2(\Omega)))} +
      \tau \nm{ E(t_j)-\mathcal E(t_j-) }_{
        L^1(t_{j-1},t_j;\mathcal L(L^2(\Omega)))
      }
    \Big) \\
    \leqslant{} & C_{\omega_0,\mathcal M_0} \tau^\alpha
    \Big(
      \alpha^{-1} +
      (1-J^{\alpha-1})(1-\alpha)^{-1}
    \Big).
  \end{align*}
  This proves \eqref{eq:int-E-wtE} and hence this lemma.
\end{proof}

Finally, we are in a position to conclude the proof of Theorem \ref{thm:conv-Stau} as
follows. By \eqref{eq:Sg-l1}  and \eqref{eq:Stau-g} we have
\begin{align*}
  \max_{1 \leqslant j \leqslant J}
  \nm{(Sg)(t_j)-S_{\tau,j}g}_{L^2(\Omega)} \leqslant
  \nm{E-\mathcal E}_{L^1(0,T;\mathcal L(L^2(\Omega)))}
  \nm{g}_{L^\infty(0,T;L^2(\Omega))},
\end{align*}
so that \eqref{eq:S-Stau-g} follows from \eqref{eq:int-E-wtE}. By
\eqref{eq:calE-def} we see that $ \mathcal E $ is piecewise constant, and then by
\eqref{eq:Stau-delta}, \eqref{eq:Stau-g} and \eqref{eq:delta_0} we obtain $ S_{\tau,j}(v\delta_0) =
\mathcal E(t_j-) v $, $ 1 \leqslant j \leqslant J $. Hence, a straightforward
computation yields, by \eqref{eq:Sdelta},
\begin{small}
\begin{align*}
  & \max_{1 \leqslant j \leqslant J} j^{2-\alpha} \nm{
    S(v\delta_0)(t_j) \!-\! S_{\tau,j}(v\delta_0)
  }_{L^2(\Omega)} \! \leqslant \!
  \max_{1 \leqslant j \leqslant J} j^{2-\alpha}
  \nm{E(t_j) \!-\! \mathcal E(t_j\!-)}_{\mathcal L(L^2(\Omega)\!)}
  \nm{v}_{L^2(\Omega)}, \\
  & \sum_{j=1}^J \nm{
    S(v\delta_0) - S_{\tau,j}(v\delta_0)
  }_{ L^1(t_{j-1},t_j;L^2(\Omega)) } \leqslant
  \nm{E-\mathcal E}_{L^1(0,T;\mathcal L(L^2(\Omega)))}
  \nm{v}_{L^2(\Omega)}.
\end{align*}
\end{small}
Therefore, \eqref{eq:S-Stau-vdelta}, \eqref{eq:S-Stau-vdelta-l1} follow from
\eqref{eq:E-calE}, \eqref{eq:int-E-wtE}, respectively. This completes the proof of
Theorem \ref{thm:conv-Stau}.

\section{An inverse problem of a fractional diffusion equation}
\label{sec:inverse}
\subsection{Continuous problem} 
We consider reconstructing the source term of a fractional diffusion equation
from the value of the solution at a fixed time; more precisely, the task is to
seek a suitable source $ f $ to ensure that the solution of problem
\eqref{eq:model} achieves a given value $ y_d $ at the final time $ T $. Applying
the well-known Tikhonov regularization technique to this inverse problem yields
the following minimization problem:
\begin{equation}
  \label{eq:inverse}
  \min\limits_{
    \substack{
      u \in U_{\text{ad}} \\
      y \in C((0,T];L^2(\Omega))
    }
  } J(y,u) := \frac12 \nm{y(T) - y_d}_{L^2(\Omega)}^2 +
  \frac\nu2 \nm{u}_{L^2(0,T;L^2(\Omega))}^2,
\end{equation}
subject to the state equation
\begin{equation}
  (\D_{0+}^\alpha - \mathcal A) y = u,
  \quad\text{ with } y(0) = 0,
\end{equation}
where $ y_d \in L^2(\Omega) $, $ \nu > 0 $ is a regularization parameter, and
\begin{align*}
  U_{\text{ad}} &:= \left\{
    v \in L^2(0,T;L^2(\Omega)):\
    u_* \leqslant v \leqslant u^* \text{ a.e.~in } \Omega \times (0,T)
  \right\},
\end{align*}
with $ u_* $ and $ u^* $ being two given constants.
\begin{remark}
  We refer the reader to \cite{Prilepko2000,Samarskii2007} for the inverse
  problems of parabolic partial differential equations, and refer the reader to
  \cite[Chapter 3]{Troltzsh2010} for the linear-quadratic parabolic control
  problems.
\end{remark}

We call $ u \in U_\text{ad} $ a mild solution to problem \eqref{eq:inverse} if $
u $ solves the following minimization problem:
\begin{equation}
  \label{eq:mild_sol}
  \min_{u \in U_\text{ad}}
  J(u) := \frac12 \nm{(Su)(T) - y_d}_{L^2(\Omega)}^2 +
  \frac\nu2 \nm{u}_{L^2(0,T;L^2(\Omega))},
\end{equation}
where we recall that $ S $ is defined by \eqref{eq:Sg-l1}.

\begin{lemma} 
  \label{lem:Sg-dual-weakly}
  Assume that $ g \in L^q(0,T;L^2(\Omega)) $ with $ q > 1/\alpha $. Then
  \begin{equation}
    \label{eq:Sdelta-dual}
    \dual{(Sg)(T), v}_\Omega =
    \dual{g, S^*(v\delta_T)}_{\Omega \times (0,T)}
  \end{equation}
  for all $ v \in L^2(\Omega) $.
\end{lemma}
\begin{proof}
  By \eqref{eq:Sg-l1} and \eqref{eq:Sg-C}, $ Sg \in C([0,T];L^2(\Omega)) $ and
  \begin{small}
  \[
    (Sg)(T) = \int_0^T E(T-t) g(t) \, \mathrm{d}t,
  \]
  \end{small}
  so that
  \begin{small}
  \[ 
    \dual{(Sg)(T), v}_\Omega =
    \Dual{\int_0^T E(T-t) g(t) \, \mathrm{d}t, \, v }_\Omega =
    \int_0^T \dual{E(T-t)g(t), v}_\Omega \, \mathrm{d}t.
  \]
  \end{small}
  Because \eqref{eq:E-E*} implies
  \[
    \dual{E(T-t)g(t), v}_\Omega = \dual{g(t),E^*(T-t)v}_\Omega,
    \quad \text{a.e.}~t \in (0,T),
  \]
  it follows that
  \begin{align*} 
    \dual{(Sg)(T), v}_\Omega & =
    \int_0^T \dual{g(t), E^*(T-t) v}_\Omega \, \mathrm{d}t
 = \dual{g, S^*(v\delta_T)}_{\Omega \times (0,T)}
    \quad\text{(by \eqref{eq:S*delta}),}
  \end{align*}
  namely, \eqref{eq:Sdelta-dual} holds indeed. This completes the proof.
\end{proof}

Assume that $ q > 1/\alpha $ and $ q \geqslant 2 $. By \eqref{eq:Sg-C}, $
(S\cdot)(T) $ is a bounded linear operator form $ L^q(0,T;L^2(\Omega)) $ to $
L^2(\Omega) $. Clearly, $ J $ in \eqref{eq:mild_sol} is a strictly convex
functional on $ L^q(0,T;L^2(\Omega)) $, and $ U_\text{ad} $ is a convex, bounded
and closed subset of $ L^q(0,T;L^2(\Omega)) $. By Lemma \ref{lem:Sg-dual-weakly}, a
routine argument (cf.~\cite[Theorems 2.14 and 2.21]{Troltzsh2010}) yields the
following theorem.
\begin{theorem} 
  \label{thm:basic-regu} Problem \eqref{eq:mild_sol} admits a unique mild
  solution $ u \in U_\text{ad} $, and the following first-order optimality
  condition holds:
  \begin{subequations}
  \begin{numcases}{}
    y = Su, \label{eq:optim-y} \\
    p = S^*\big( (y(T)-y_d)\delta_T \big), \label{eq:optim-p} \\
    \Dual{p + \nu u, v-u}_{\Omega \times (0,T)}
    \geqslant 0 \quad \text{ for all } v \in U_\text{ad}.
    \label{eq:optim-u}
  \end{numcases}
  \end{subequations}
\end{theorem}

\begin{remark} 
  Assume that $ u $, $ y $ and $ p $ are defined in Theorem \ref{thm:basic-regu}. By
  (\ref{eq:optim-u}) we have $ u = f(p) $, where
  \begin{small}
  \begin{equation}
    \label{eq:f}
    f(r) := \begin{cases}
      u_* & \text{ if } r > -\nu u_*, \\
      r & \text{ if } -\nu u^* \leqslant r \leqslant -\nu u_*, \\
      u^* & \text{ if } r < -\nu u^*.
    \end{cases}
  \end{equation}
  \end{small}
  Noting that $ f $ is Lipschitz continuous with Lipschitz constant $ 1/\nu $,
  we obtain
  \begin{small}
  \[
    u'(t) = f'(p(t)) p'(t) \text{ in } L^2(\Omega),
    \quad \text{a.e.}~0 < t < T,
  \]
  \end{small}
  and hence $ \nm{u'(t)}_{L^2(\Omega)} \leqslant \nu^{-1}
  \nm{p'(t)}_{L^2(\Omega)} $, a.e.~$ 0 < t < T $. It follows from
  (\ref{eq:optim-p}), \eqref{eq:S*delta} and \eqref{eq:E*'} that
  \[
    \nm{u'(t)}_{L^2(\Omega)} \leqslant
    C_{\omega_0,\mathcal M_0} \nu^{-1} (T-t)^{\alpha-2}
    (\nm{y(T)}_{L^2(\Omega)} + \nm{y_d}_{L^2(\Omega)}),
    \quad\text{a.e.}~0 < t < T.
  \]
  Since (\ref{eq:optim-y}), \eqref{eq:Sg-l1}, \eqref{eq:E} and the fact $ u \in
  U_\text{ad} $ imply
  \begin{equation}
    \label{eq:yT}
    \nm{y(T)}_{L^2(\Omega)} \leqslant
    C_{u_*,u^*,\omega_0,\mathcal M_0,T,\Omega}
    \alpha^{-1},
  \end{equation}
  we conclude therefore that
  \begin{small}
  \begin{equation}
    \label{eq:u}
    \nm{u'(t)}_{L^2(\Omega)} \leqslant
    C_{u_*,u^*,\omega_0,\mathcal M_0,T,\Omega} \nu^{-1}
    (T-t)^{\alpha-2}(\alpha^{-1} + \nm{y_d}_{L^2(\Omega)}),
    \quad\text{a.e.}~0 < t < T.
  \end{equation}
  \end{small}
\end{remark}
\begin{remark}
  \label{rem:nu=0}
  Let $ u_\nu $ be the mild solution of problem \eqref{eq:mild_sol}. A standard
  argument yields that there exits $ y_T \in L^2(\Omega) $ such that
  \begin{equation}
    \label{eq:731}
    \nm{(Su_\nu)(T)-y_T}_{L^2(\Omega)} \leqslant
    C_{u_*,u^*,T,\Omega} \sqrt\nu.
  \end{equation}
  Since $ U_\text{ad} $ is a convex, bounded and closed subset of $
  L^q(0,T;L^2(\Omega)) $, $ q > 1/\alpha $, there exist $ u_0 \in U_\text{ad} $
  and a decreasing sequence $ \{\nu_n\}_{n=0}^\infty \subset (0,\infty) $ with
  limit zero such that
  \[
    \lim_{n \to \infty} u_{\nu_n} = u_0 \quad\text{ weakly in }
    L^q(0,T;L^2(\Omega)).
  \]
  As $ (S \cdot)(T) $ is a bounded linear operator from $ L^q(0,T;L^2(\Omega)) $
  to $ L^2(\Omega) $, we have that $ (Su_{\nu_n})(T) $ converges to $ (Su_0)(T)
  $ weakly in $ L^2(\Omega) $ as $ n \to \infty $, so that \eqref{eq:731} implies
  $ (Su_0)(T) = y_T $. Furthermore, a trivial calculation yields that $ u_0 $ is
  a mild solution of problem \eqref{eq:inverse} with $ \nu=0 $.
\end{remark}

\subsection{Temporally discrete problem}
\label{ssec:discr}
Define
\[
  W_\tau := \{
    V \in L^\infty(0,T; H_0^1(\Omega)):\,
    V \text{ is constant on } (t_{j-1},t_j)
    \quad \forall 1 \leqslant j \leqslant J
  \}.
\]
For any $ g \in W_\tau^* $, define $ S_\tau g \in W_{\tau} $ and $ S_{\tau}^*g
\in W_{\tau} $, respectively, by that
\begin{align}
  \dual{ \, \D_{0+}^\alpha S_\tau g, V}_{\Omega \times (0,T)} -
  \dual{\mathcal AS_\tau g, V}_{L^2(0,T; H_0^1(\Omega))} =
  \dual{g,V}_{W_\tau}, \label{eq:Stau} \\
  \dual{( \, \D_{T-}^\alpha S_\tau^* g, V}_{\Omega \times (0,T)} -
  \dual{\mathcal A^*S_\tau^* g, V}_{L^2(0,T; H_0^1(\Omega))} =
  \dual{g, V}_{W_\tau},
  \label{eq:S*tau}
\end{align}
for all $ V \in W_{\tau} $. By \eqref{eq:dual} we have that
\begin{equation}
  \label{eq:Stau-dual}
  \dual{S_{\tau}f, g}_{\Omega \times (0,T)} =
  \dual{f, S_{\tau}^*g}_{\Omega \times (0,T)}
  \quad \forall f, g \in L^1(0,T;L^2(\Omega)).
\end{equation}
A direct calculation yields that (cf.~\cite[Remark 3]{Jin-maximal2018}), for any
$ g \in W_\tau^* $,
\begin{equation}
  \label{eq:Stau-Stauj}
  (S_\tau g)(t_j-) = S_{\tau, j} g \quad \forall 1 \leqslant j \leqslant J.
\end{equation}
Hence, from Theorem \ref{thm:conv-Stau}, we readily conclude the following two
estimates: for any $ g \in L^\infty(0,T;L^2(\Omega)) $,
\begin{equation}
  \label{eq:S-Stau-g-2}
  \nm{(Sg)(T) - (S_\tau g)(T-)}_{L^2(\Omega)}
  \leqslant C_{\omega_0, \mathcal M_0} \tau^\alpha
  \Big(
    \frac1\alpha + \frac{1-J^{\alpha-1}}{1-\alpha}
  \Big) \nm{g}_{L^\infty(0,T;L^2(\Omega))};
\end{equation}
for any $ v \in L^2(\Omega) $,
\begin{equation}
  \label{eq:S-Stau-delta}
  \nm{S(v\delta_0) - S_\tau(v\widehat\delta_0)}_{L^1(0,T;L^2(\Omega))}
  \leqslant C_{\omega_0,\mathcal M_0} \tau^\alpha
  \Big(
    \frac1\alpha + \frac{1-J^{\alpha-1}}{1-\alpha}
  \Big) \nm{v}_{L^2(\Omega)}.
\end{equation}
Furthermore, we have the following stability estimate.
\begin{lemma}
  \label{lem:Stau-stab-infty}
  Assume that $ g \in {}_0H^{-\alpha/2}(0,T;L^2(\Omega)) $. Then, for any $ 1
  \leqslant j \leqslant J $,
  \begin{equation}
    \label{eq:Stau-stab-infty}
    \nm{(S_\tau g)(t_j-)}_{L^2(\Omega)} \leqslant
    C_\alpha \tau^{(\alpha-1)/2}
    \nm{g}_{{}_0H^{-\alpha/2}(0,T;L^2(\Omega))}.
  \end{equation}
\end{lemma}
\begin{proof}
  We only prove \eqref{eq:Stau-stab-infty} with $ j=J $, since the other cases $
  1 \leqslant j < J $ can be proved analogously. Let $ v:= (S_\tau g)(t_j-) $.
  We have
  \begin{align*}
    & \nm{v}_{L^2(\Omega)}^2 =
    \dual{v\widehat\delta_T, S_\tau g}_{\Omega \times (0,T)} \\
    ={} & \dual{
      \D_{T-}^{\alpha/2} \D_{T-}^{-\alpha/2} (v\widehat\delta_T),
      S_\tau g
    }_{\Omega \times (0,T)} \\
    ={} &
    \dual{
      \D_{T-}^{-\alpha/2}(v\widehat\delta_T),
      \D_{0+}^{\alpha/2} S_\tau g
    }_{\Omega \times (0,T)} \quad\text{(by \eqref{eq:dual})} \\
    \leqslant{} & \nm{\D_{0+}^{\alpha/2} S_\tau g}_{L^2(0,T;L^2(\Omega))}
    \nm{\D_{T-}^{-\alpha/2}(v\widehat\delta_T)}_{L^2(0,T;L^2(\Omega))},
  \end{align*}
  where we recall that $ \widehat\delta_T $ is defined by \eqref{eq:delta_T}.
  Since inserting $ V:= S_\tau g $ into \eqref{eq:Stau} yields, by
  \eqref{eq:dual}, \eqref{eq:coer} and \eqref{eq:A-positive}, that
  \[
    \nm{\D_{0+}^{\alpha/2}S_\tau g}_{L^2(0,T;L^2(\Omega))}
    \leqslant C_{\alpha} \nm{g}_{{}_0H^{-\alpha/2}(0,T;L^2(\Omega))},
  \]
  it follows that
  \[
    \nm{v}_{L^2(\Omega)}^2
    \leqslant
    C_\alpha
    \nm{g}_{{}_0H^{-\alpha/2}(0,T;L^2(\Omega))}
    \nm{\D_{T-}^{-\alpha/2}(v\widehat\delta_T)}_{L^2(\Omega)}.
  \]
  It suffices, therefore, to prove
  \begin{equation}
    \label{eq:frac-z}
    \nm{\D_{T-}^{-\alpha/2}(v\widehat\delta_T)}_{L^2(0,T;L^2(\Omega))}
    \leqslant C_\alpha \tau^{(\alpha-1)/2} \nm{v}_{L^2(\Omega)}.
  \end{equation}

  To this end, we note that
  \begin{align*}
    & \nm{\D_{T-}^{-\alpha/2}(v\widehat\delta_T)}_{L^2(0,T;L^2(\Omega))}^2 \\
    ={}& \left(
      \frac{\nm{v}_{L^2(\Omega)}}{\Gamma(\alpha/2)}
    \right)^2 \tau^{-2} \int_0^T \Snm{
      \int_t^T (s-t)^{\alpha/2-1} \widehat\delta_T(s) \, \mathrm{d}s
    }^2 \, \mathrm{d}t \\
    ={}& \left(
      \frac{\nm{v}_{L^2(\Omega)}}{\Gamma(\alpha/2)}
    \right)^2 \tau^{-2} (\mathbb I_1 + \mathbb I_2),
  \end{align*}
  where
  \begin{small}
  \begin{align*}
    \mathbb I_1 &:= \int_0^{T-\tau}  \Snm{
      \int_{T-\tau}^T (s-t)^{\alpha/2-1} \, \mathrm{d}s
    }^2 \, \mathrm{d}s \, \mathrm{d}t, \\
    \mathbb I_2 &:= \int_{T-\tau}^T \Snm{
      \int_t^T (s-t)^{\alpha/2-1} \, \mathrm{d}s
    }^2 \, \mathrm{d}t.
  \end{align*}
  \end{small}
  A straightforward calculation gives
  \begin{align*}
    \mathbb I_1 &= 4/\alpha^2 \int_0^{T-\tau} \big(
      (T-t)^{\alpha/2} - (T-\tau-t)^{\alpha/2}
    \big)^2 \, \mathrm{d}t \\
    &= 4/\alpha^2 \tau^{1+\alpha}
    \int_0^{T/\tau} \big(
      s^{\alpha/2} - (s-1)^{\alpha/2}
    \big)^2 \, \mathrm{d}s \\
    &< 4/\alpha^2 \tau^{1+\alpha}
    \int_0^\infty \big(
      s^{\alpha/2} - (s-1)^{\alpha/2}
    \big)^2 \, \mathrm{d}s  = C_\alpha \tau^{1+\alpha}
  \end{align*}
  and
  \begin{align*}
    \mathbb I_2 = 4/\alpha^2 \int_{T-\tau}^T
    (T-t)^\alpha \, \mathrm{d}t =
    C_\alpha \tau^{1+\alpha}.
  \end{align*}
  Combining the above estimates of $ \mathbb I_1 $ and $ \mathbb I_2 $ proves
  \eqref{eq:frac-z} and hence this lemma.
\end{proof}

\begin{remark}
  We note that if the temporal grid is nonuniform, then \eqref{eq:Stau} is not
  equivalent to the L1 scheme for fractional diffusion equations. For the
  numerical analysis of \eqref{eq:Stau} with nonuniform temporal grid, we refer
  the reader to \cite{Li2019SIAM,Li-Wang-Xie2020}.
\end{remark}


Following the idea in \cite{Hinze2005}, we consider the following temporally
discrete problem:
\begin{equation}
  \label{eq:numer_opti}
  \min\limits_{U \in U_{\text{ad}}} J_{\tau}(U) :=
  \frac12 \nm{ ( S_{\tau}U)(T-) - y_d }_{L^2(\Omega)}^2 +
  \frac\nu2 \nm{U}_{L^2(0,T;L^2(\Omega))}^2.
\end{equation}
Note that $ U_\text{ad} $ is a convex, bounded and closed subset of $
L^2(0,T;L^2(\Omega)) $ and that $ (S_\tau\cdot)(T-) $ is, by
\eqref{eq:Stau-stab-infty}, a bounded linear operator from $ L^2(0,T;L^2(\Omega))
$ to $ L^2(\Omega) $. Hence, applying \cite[Theorems 2.14 and
2.21]{Troltzsh2010} to problem \eqref{eq:numer_opti} yields the following
theorem.
\begin{theorem}
  \label{thm:regu-U}
  Problem \eqref{eq:numer_opti} admits a unique solution $ U \in U_\text{ad} $,
  and the following optimality condition holds:
  \begin{subequations}
  \begin{numcases}{}
    Y = S_\tau U, \label{eq:optim-Y} \\
    P = S_\tau^*\big( (Y(T-)-y_d)\widehat\delta_T \big), \label{eq:optim-P} \\
    \Dual{P + \nu U, V-U}_{\Omega \times (0,T)}
    \geqslant 0 \quad \text{ for all } V \in U_\text{ad},
    \label{eq:optim-U}
  \end{numcases}
  \end{subequations}
  where $ \widehat\delta_T $ is defined by \eqref{eq:delta_T}.
\end{theorem}

\begin{theorem} 
  \label{thm:conv}
  Let $ u $ and $ y $ be defined in Theorem \ref{thm:basic-regu}, and let $ U $ and $ Y
  $ be defined in Theorem \ref{thm:regu-U}. Then
  \begin{equation} 
    \label{eq:conv}
    \begin{aligned}
      & \nm{(y-Y)(T-)}_{L^2(\Omega)} +
      \sqrt\nu \nm{u-U}_{L^2(0,T;L^2(\Omega))} \\
      \leqslant{} &
      C_{y_d,u_*,u^*,\omega_0,\mathcal M_0,T,\Omega}
      \left(
        \frac1\alpha + \left(\frac{1-J^{\alpha-1}}{1-\alpha}\right)^{1/2} +
        \frac{1-J^{\alpha-1}}{1-\alpha} \tau^{\alpha/2}
      \right) \tau^{\alpha/2}.
    \end{aligned}
  \end{equation}
\end{theorem}
\begin{proof} 
  Since the idea of of this proof is standard (cf.~\cite[Theorem
  3.4]{Hinze2009}), we only provide a brief proof. Let us first prove that
  \begin{align}
    & \nm{(Su)(T) - (S_\tau U)(T-)}_{L^2(\Omega)} +
    \sqrt\nu \nm{u-U}_{L^2(0,T;L^2(\Omega))} \notag \\
    \leqslant{} &
    C_{u_*,u^*,\Omega} \nm{
      S^*((y(T)-y_d)\delta_T) -
      S_{\tau}^*((y(T)-y_d)\widehat\delta_T)
    }_{L^1(0,T;L^2(\Omega))}^{1/2} \notag \\
    & \quad {} + 2\nm{(Su)(T) - (S_{\tau}u)(T-)}_{L^2(\Omega)}.
    \label{eq:ava}
  \end{align}
  By \eqref{eq:optim-u} and \eqref{eq:optim-U}, we have
  \begin{align*} 
    \Dual{
      S^*\big( (y(T) - y_d)\delta_T \big) + \nu u, U-u
    }_{\Omega \times (0,T)} \geqslant 0,\\
    \Dual{
      S_{\tau}^* \big( ( Y(T-) - y_d) \widehat\delta_T \big) + \nu U, u-U
    }_{\Omega \times (0,T)} \geqslant 0,
  \end{align*}
  so that
  \begin{equation}
    \label{eq:I1+I2}
    \nu \nm{u-U}_{L^2(0,T;L^2(\Omega))}^2
    \leqslant \mathbb I_1 + \mathbb I_2,
  \end{equation}
  where
  \begin{align*} 
    \mathbb I_1 &:= \dual{
      S^* \big( (y(T)-y_d)\delta_T \big) -
      S_{\tau}^* \big( y(T)-y_d)\widehat\delta_T \big),
      U-u
    }_{\Omega \times (0,T)}, \\
    \mathbb I_2 &:=
    \dual{
      S_{\tau}^* \big( (y(T)  - Y(T-))\widehat\delta_T \big),
      U-u
    }_{\Omega \times (0,T)}.
  \end{align*}
  It is clear that
  \[
    \mathbb I_1 \leqslant C_{u_*,u^*,\Omega}
    \nm{
      S^*((y(T)-y_d)\delta_T) -
      S_\tau^*((y(T)-y_d)\widehat\delta_T)
    }_{L^1(0,T;L^2(\Omega))},
  \]
  by the fact that $ u, U \in U_\text{ad} $. A straightforward computation
  yields
  \begin{align*}
    \mathbb I_2 & = \dual{
      (y(T)  - Y(T-))\widehat\delta_T,
      S_\tau(U-u)
    }_{\Omega \times (0,T)} \quad\text{(by \eqref{eq:Stau-dual})} \\
    &= \dual{ y(T)-Y(T-), (S_{\tau}(U-u))(T-) }_{\Omega}
    \quad\text{(by \eqref{eq:delta_T})} \\
    &= \dual{ (Su)(T)-(S_\tau U)(T-), (S_{\tau}(U-u))(T-) }_{\Omega}
    \quad \text{(by (\ref{eq:optim-y}) and (\ref{eq:optim-Y}))} \\
    &= \dual{
      (Su)(T) - (S_{\tau}u)(T-),
      (S_{\tau}(U-u))(T-)
    }_\Omega -
    \nm{(S_{\tau}(u-U))(T-)}_{L^2(\Omega)}^2 \\
    & \leqslant
    \frac12 \nm{(Su)(T) - (S_{\tau}u)(T-)}_{L^2(\Omega)}^2 -
    \frac12 \nm{(S_{\tau}(u-U))(T-)}_{L^2(\Omega)}^2 \\
    & \leqslant
    \nm{(Su)(T) - (S_{\tau}u)(T-)}_{L^2(\Omega)}^2 -
    \frac12 \nm{(Su)(T) - (S_\tau U)(T-)}_{L^2(\Omega)}^2.
  \end{align*}
  Combining \eqref{eq:I1+I2} and the above estimates of $ \mathbb I_1 $ and $
  \mathbb I_2 $ gives \eqref{eq:ava}.

  Then, by the symmetric version of \eqref{eq:S-Stau-delta} we obtain
  \begin{align*}
    & \nm{
      S^*((y(T)-y_d)\delta_T) -
      S_{\tau}^*((y(T)-y_d)\widehat\delta_T)
    }_{L^1(0,T;L^2(\Omega))} \\
    \leqslant{} &
    C_{\omega_0,\mathcal M_0} \tau^\alpha
    \left(
      \frac1\alpha + \frac{1-J^{\alpha-1}}{1-\alpha}
    \right) \nm{y(T) - y_d}_{L^2(\Omega)},
  \end{align*}
  so that \eqref{eq:yT} implies
  \begin{align}
    & \nm{
      S^*((y(T)-y_d)\delta_T) - S_{\tau}^*((y(T)-y_d)\widehat\delta_T)
    }_{L^1(0,T;L^2(\Omega))} \notag \\
    \leqslant{} &
    C_{u_*,u^*,\omega_0,\mathcal M_0,T,\Omega} \tau^\alpha
    \left(
      \frac1\alpha + \frac{1-J^{\alpha-1}}{1-\alpha}
    \right) (1/\alpha +  \nm{y_d}_{L^2(\Omega)}) \notag \\
    \leqslant{} &
    C_{y_d,u_*,u^*,\omega_0,\mathcal M_0,T,\Omega}
    \left(
      \frac1{\alpha^2} + \frac{1-J^{\alpha-1}}{1-\alpha}
    \right) \tau^\alpha.
    \label{eq:eve-1}
  \end{align}
  We obtain from \eqref{eq:S-Stau-g-2} that
  \begin{equation}
    \label{eq:eve-2}
    \nm{(Su)(T) - (S_\tau u)(T-)}_{L^2(\Omega)}
    \leqslant C_{u_*,u^*,\omega_0,\mathcal M_0,\Omega}
    \left(
      \frac1\alpha + \frac{1-J^{\alpha-1}}{1-\alpha}
    \right) \tau^\alpha.
  \end{equation}
  Finally, combining \eqref{eq:ava}, \eqref{eq:eve-1} with \eqref{eq:eve-2} gives
  \begin{align*}
    & \nm{(Su)(T) - (S_\tau U)(T-)}_{L^2(\Omega)} +
    \sqrt\nu \nm{u-U}_{L^2(0,T;L^2(\Omega))} \\
    \leqslant{} &
    C_{y_d,u_*,u^*,\omega_0,\mathcal M_0,T,\Omega}
    \left(
      \frac1\alpha + \left(\frac{1-J^{\alpha-1}}{1-\alpha}\right)^{1/2} +
      \frac{1-J^{\alpha-1}}{1-\alpha} \tau^{\alpha/2}
    \right) \tau^{\alpha/2},
  \end{align*}
  which, together with (\ref{eq:optim-y}) and (\ref{eq:optim-Y}), implies
  \eqref{eq:conv}. This completes the proof.
\end{proof}
\begin{remark}
  Let $ y_T $ be defined in Remark \ref{rem:nu=0}. Combining \eqref{eq:731} and \eqref{eq:conv}
  yields
  \begin{small}
  \begin{align*}
    & \nm{y_T - Y({T-})}_{L^2(\Omega)} \\
    \leqslant{} &
    C_{y_d,u_*,u^*,\omega_0,\mathcal M_0,T,\Omega} \left(
      \sqrt\nu +
      \left(
        \frac1\alpha + \left(\frac{1-J^{\alpha-1}}{1-\alpha}\right)^{1/2} +
        \frac{1-J^{\alpha-1}}{1-\alpha} \tau^{\alpha/2}
      \right) \tau^{\alpha/2}
    \right).
  \end{align*}
  \end{small}
\end{remark}

\section{Numerical experiments} 
\label{sec:numer} 
This section performs three numerical experiments in one dimensional space to
verify the theoretical results, in the following settings: $ T= 0.1 $; $ \Omega
= (0,1) $; $ \mathcal A = \Delta $; the space is discretized by a standard
Galerkin finite element method, with the space
\[ 
  \mathcal V_h := \left\{
    v_h \in H_0^1(0,1): v_h \text{ is linear on }
    \big( (m\!-\!1)/2^{10}, m/2^{10} \big) \,\,
    \text{for all $ 1 \leqslant m \leqslant 2^{10} $}
  \right\}.
\]

\noindent{\it Experiment 1.} The purpose of this experiment is to verify
\eqref{eq:S-Stau-vdelta} and \eqref{eq:S-Stau-vdelta-l1}. We set $ v(x) := x^{-0.49} $, $ 0 <
x < 1 $, and let
\begin{align*}
  e_T &:= \nm{
    S_{\tau,J}(v\delta_0) -
    S_{\tau^*,J^*}(v\delta_0)
  }_{L^2(\Omega)}, \\
  e_{l1} &:= \sum_{j=1}^{J^*} T/J^*
  \nm{
    S_{\tau, \left\lceil j J/J^* \right\rceil}(v\delta_0) -
    S_{\tau^*, j} (v\delta_0)
  }_{L^2(\Omega)},
\end{align*}
where $ J^* := 2^{15} $, $ \tau^* = T/J^* $, and $ \lceil \cdot \rceil $ is the
ceiling function. Table \ref{tab:Svdelta-alpto1} shows that $
e_T/(\tau^{\alpha-1}J^{\alpha-2}) $ will not blow up as $ \alpha \to {1-} $,
which agrees well with \eqref{eq:S-Stau-vdelta}. The numerical results in Figure
\ref{fig:ex1} illustrate that $ e_T $ is close to $ O(\tau) $, and this also agrees
well with \eqref{eq:S-Stau-vdelta}. The numerical results in Figure
\ref{fig:ex1_l1} demonstrate that $ e_{l1} $ is close to $ O(\tau^\alpha) $,
and this is in good agreement with \eqref{eq:S-Stau-vdelta-l1}.

\begin{table}
  \caption{$ e_T / (\tau^{\alpha-1} J^{\alpha-2}) $ of Experiment 1.}
  \label{tab:Svdelta-alpto1}
\begin{center}
  \begin{tabular}{cccc}
    \toprule
    $\alpha$ & $J=2^7$ & $J=2^8$ & $J=2^9$ \\
    $0.90$   & 5.35e-3 & 5.19e-3 & 5.03e-3 \\
    $0.95$   & 5.13e-3 & 4.90e-3 & 4.74e-3 \\
    $0.99$   & 4.37e-3 & 4.10e-3 & 3.94e-3 \\
    $0.999$  & 4.10e-3 & 3.82e-3 & 3.66e-3 \\
    \bottomrule
  \end{tabular}
\end{center}
\end{table}

\begin{figure}[H]
  \begin{minipage}[t]{0.5\linewidth}
    \includegraphics[scale=0.5]{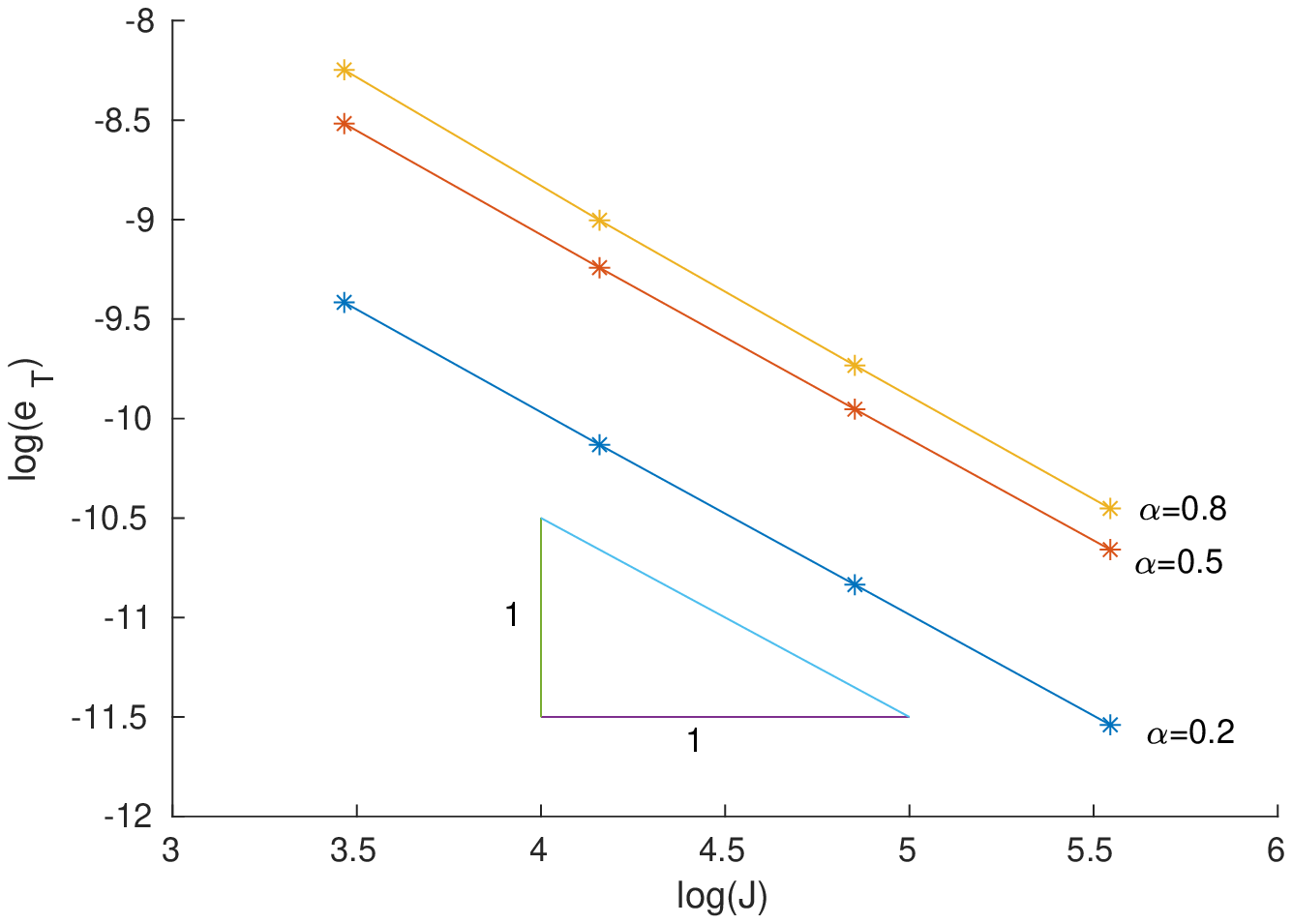}
    \caption{$e_T$ of numerical Example 1.}
    \label{fig:ex1}
  \end{minipage}%
  \begin{minipage}[t]{0.5\linewidth}
    \includegraphics[scale=0.5]{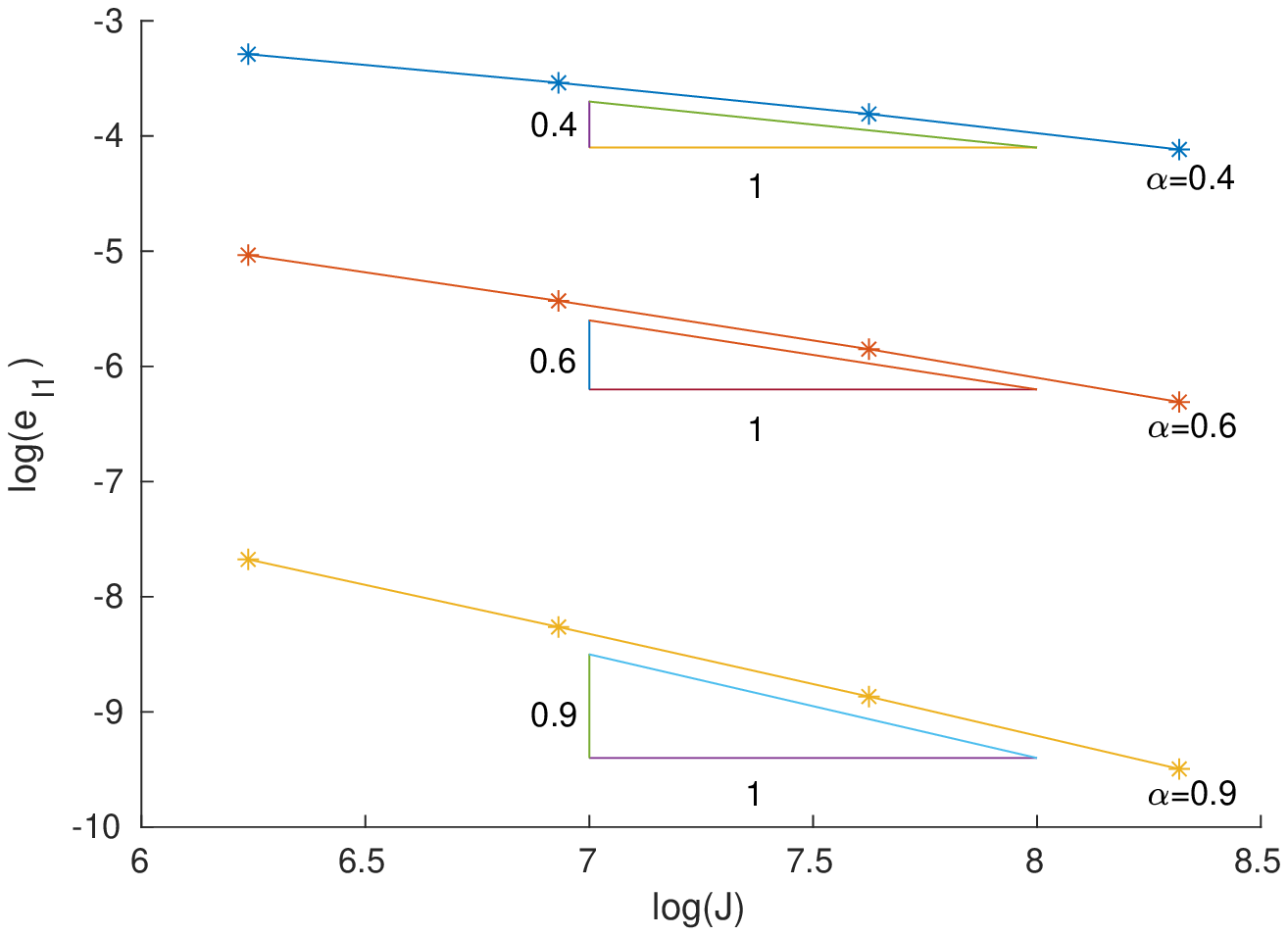}
    \caption{$e_{l1}$ of numerical Example 1.}
    \label{fig:ex1_l1}
  \end{minipage}
\end{figure}

\noindent{\it Experiment 2.} The purpose of this experiment is to verify
\eqref{eq:S-Stau-g}. To this end, we set
\[
  f(t,x) := x^{-0.49}, \quad 0 < t < T, \quad 0 < x < 1,
\]
and define
\[
  e_\infty := \max_{1 \leqslant j \leqslant J}
  \nm{
    S_{\tau,j}f - S_{\tau^*,\lceil jJ^*/J \rceil} f
  }_{L^2(\Omega)},
\]
where $ J^* = 2^{15} $ and $ \tau^* = T/J^* $. The numerical results in
Figure \ref{fig:ex2} shows that $ e_\infty $ is close to $ O(\tau^\alpha) $,
which is in good agreement with \eqref{eq:S-Stau-g}.
\begin{figure}[H]
    \centering
    \includegraphics[scale=.5]{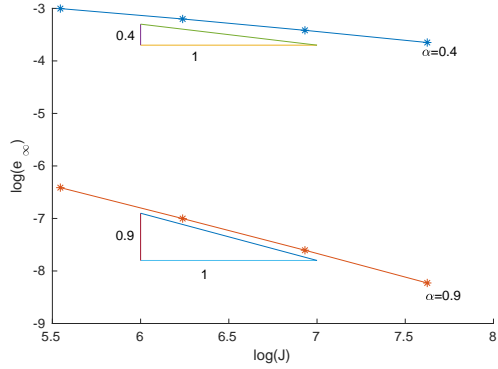}
    \caption{$e_\infty$ of numerical Example 2.}
    \label{fig:ex2}
\end{figure}


\medskip\noindent{\it Experiment 3.} The purpose of this experiment is to verify
Theorem \ref{thm:conv}, in the following settings: $ a=0 $; $ b=10 $; $ \nu = 10 $; $
y_d :\equiv 1 $. Discretization \eqref{eq:numer_opti} is solved by the following
iteration algorithm (cf.~\cite[Algorithm 3.2]{Hinze2009}):
\begin{small}
\begin{align*}
  & U_0 := 0, \\
  & U_j = f(S_\tau^*(((S_\tau U_{j-1})(T-)-y_d)\widehat\delta_T)),
  \quad 1 \leqslant j \leqslant k,
\end{align*}
\end{small}
where $ f $ is defined by \eqref{eq:f} and $ k $ is large enough such that
\[
  \nm{U_k - U_{k-1}}_{L^\infty(0,T;L^\infty(\Omega))} < 10^{-12}.
\]
The ``Error" in  Figure \ref{fig:ex3} means
\[
  \nm{Y(T-)-Y^*(T-)}_{L^2(\Omega)} +
  \nm{U - U^*}_{L^2(0,T;L^2(\Omega))},
\]
where $ U^* $ and $ Y^* $ are the numerical solutions with $ J =2^{15} $. The
theoretical convergence rate $ O(\tau^{\alpha/2}) $ is observed in Table
\ref{fig:ex3}.

\begin{figure}[H]
    \centering
    \includegraphics[scale=.5]{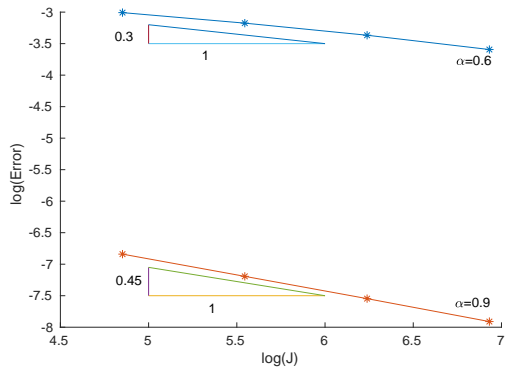}
    \caption{Numerical results of numerical Example 3.}
    \label{fig:ex3}
\end{figure}




\bibliographystyle{plain}

\section{Appendix:   Properties of the fractional calculus operators}
Assume that $ -\infty < a < b < \infty $. Define
\begin{align*}
  {}_0H^1(a,b;L^2(\Omega)) := \{v \in H^1(a,b;L^2(\Omega)): v(a) = 0 \}, \\
  {}^0H^1(a,b;L^2(\Omega)) := \{v \in H^1(a,b;L^2(\Omega)): v(b) = 0 \},
\end{align*}
where $ H^1(a,b;L^2(\Omega)) $ is a standard vector valued Sobolev space. For
each $ 0 < \beta < 1 $, define
\begin{align*}
  {}_0H^\beta(a,b;L^2(\Omega)) &:=
  (L^2(a,b;L^2(\Omega)), {}_0H^1(a,b;L^2(\Omega)))_{\beta,2}, \\
  {}^0H^\beta(a,b;L^2(\Omega)) &:=
  (L^2(a,b;L^2(\Omega)), {}^0H^1(a,b;L^2(\Omega)))_{\beta,2},
\end{align*}
where $ (\cdot,\cdot)_{\beta,2} $ means the interpolation space defined by the $
K $-method (cf.~\cite{Lunardi2018}). In addition, we use $
{}_0H^{-\beta}(a,b;L^2(\Omega)) $ and $ {}^0H^{-\beta}(a,b;L^2(\Omega)) $ to
denote the dual spaces of $ {}^0H^\beta(a,b;L^2(\Omega)) $ and $
{}_0H^\beta(a,b;L^2(\Omega)) $, respectively.

Assume that $ 0 < \gamma < 1/2 $. For any $ v \in {}_0H^\gamma(a,b;L^2(\Omega))
$,
\begin{equation}
  \label{eq:sobolev-frac}
  C_0 \nm{v}_{{}_0H^\gamma(a,b;L^2(\Omega))} \leqslant
  \nm{\D_{a+}^\gamma v}_{L^2(a,b;L^2(\Omega))} \leqslant
  C_1 \nm{v}_{{}_0H^\gamma(a,b;L^2(\Omega))},
\end{equation}
where $ C_0 $ and $ C_1 $ are two positive constants depending only on $ \gamma
$. For any $ v \in {}_0H^\gamma(a,b;L^2(\Omega)) $ and $ w \in
{}^0H^\gamma(a,b;L^2(\Omega)) $,
\begin{small}
\begin{equation}
  \label{eq:dual}
  \dual{\D_{a+}^{2\gamma} v, w}_{{}^0H^\gamma(a,b;L^2(\Omega))} =
  \dual{\D_{a+}^\gamma v, \D_{b-}^\gamma w}_{\Omega \times (0,T)} =
  \overline{
    \dual{\D_{b-}^{2\gamma} w, v}_{{}_0H^\gamma(a,b;L^2(\Omega))}
  }.
\end{equation}
\end{small}
For any $ v \in {}_0H^\gamma(a,b;L^2(\Omega)) $,
\begin{small}
\begin{equation}
  \label{eq:coer}
  \cos(\gamma\pi) \nm{\D_{a+}^\gamma v}_{L^2(a,b;L^2(\Omega))}^2
  \leqslant \Dual{
    \D_{a+}^\gamma v, \D_{b-}^\gamma v
  }_{\Omega \times (0,T)} \leqslant
  \sec(\gamma\pi) \nm{\D_{a+}^\gamma v}_{L^2(a,b;L^2(\Omega))}^2.
\end{equation}
\end{small}
For the proof of \eqref{eq:sobolev-frac} we refer the reader to \cite[Section
3]{Luo2019}, and, for the proofs of \eqref{eq:dual} and \eqref{eq:coer}, we refer the reader
to \cite{Ervin2006}.

\end{document}